\documentclass[a4paper,12pt,reqno]{amsart}
\usepackage[T1]{fontenc}
\usepackage[utf8]{inputenc}
\usepackage[english]{babel}
\usepackage[left=25mm,right=25mm,top=35mm,bottom=35mm]{geometry}
\usepackage{amsmath,amssymb,amsthm}
\usepackage{pgfplots}
\usepackage{pgfplotstable}
\usepackage{subcaption}
\usepackage{graphicx}
\usepackage[margin=0pt,font={small}]{caption}
\usepackage{booktabs}
\usepackage{color,colortbl}
\usepackage[hidelinks]{hyperref}

\usepackage{bm}
\usepackage[mathscr]{euscript}
\usepackage{mathtools}
\usepackage{stmaryrd}

\newcommand{\eps}{\varepsilon}

\newcommand{\G}{\Gamma}
\newcommand{\N}{\mathbb{N}}

\newcommand{\R}{\mathbb{R}}
\newcommand{\V}{\mathbb{V}}
\newcommand{\X}{\mathbb{X}}
\newcommand{\Y}{\mathbb{Y}}

\newcommand{\MM}{\mathcal{M}}
\newcommand{\NN}{\mathcal{N}}

\renewcommand{\SS}{\mathcal{S}}
\newcommand{\TT}{\mathcal{T}}

\DeclareMathOperator*{\hull}{span}
\DeclareMathOperator*{\refine}{refine}

\newcommand{\enorm}[3][]{#1|\!#1|\!#1|\,#2\,#1|\!#1|\!#1|_{#3}}
\newcommand{\norm}[3][]{#1\|#2#1\|_{#3}}
\newcommand{\dpi}{\mathrm{d}\pi}
\newcommand{\dx}{\mathrm{d}x}
\newcommand{\dt}{\mathrm{d}t}

\newcommand{\reff}[2]{\stackrel{\eqref{#1}}{#2}}	
\newcommand{\refp}[2]{\stackrel{\phantom{\eqref{#1}}}{#2}}

\newcommand{\Colpts}{\mathcal{Y}} 
\newcommand{\mf}{\kappa} 
\newcommand{\indset}{\Lambda} 
\newcommand{\markindset}{\Upsilon} 
\newcommand{\marg}{{\rm K}} 
\newcommand{\rmarg}{{\rm R}} 
\newcommand{\nnu}{\boldsymbol{\nu}} 
\newcommand{\eeps}{\boldsymbol{\eps}} 
\newcommand\1{\boldsymbol{1}}
\newcommand{\LagrBasis}[2]{L_{#1}^{#2}} 
\newcommand{\LagrBasisHat}[2]{\widehat L_{#1}^{#2}} 
\newcommand{\scsol}{u_{\bullet}^{\rm SC}} 
\newcommand{\scsoldual}{z_{\bullet}^{\rm SC}} 
\newcommand{\scsolhat}{\widehat u_{\bullet}^{\rm\, SC}} 
\newcommand{\scsoldualhat}{\widehat z_{\bullet}^{\rm\, SC}} 

\pgfplotsset{
every axis/.append style={
font={\fontsize{8pt}{12pt}\selectfont},  
},
tick label style={font=\tiny},
title style={font=\tiny,yshift=-1.5ex},
xlabel style={font=\tiny,yshift=+1.0ex},
ylabel style={font=\tiny,yshift=-1.2ex},
}

\definecolor{myBrown}{rgb}{0.6 0.4 0.2}
\definecolor{myOrange}{rgb}{1.0 0.6 0.2}
\definecolor{myLightGray}{RGB}{235,235,235}
\definecolor{myViolet}{RGB}{153,50,204}
\newtheorem{theorem}{Theorem}

\newtheorem{algorithm}[theorem]{Algorithm}
\newtheorem{remark}[theorem]{Remark}

\makeatletter
\def\@seccntformat#1{%
  \protect\textup{\protect\@secnumfont
    \ifnum\pdfstrcmp{subsection}{#1}=0 \bfseries\fi
    \csname the#1\endcsname
    \protect\@secnumpunct
  }%
}
\makeatother
\usepackage{fancyhdr}
\advance\footskip0.5cm
\definecolor{otherblue}{rgb}{0,0.3,0.6}
\newcommand\revision[1]{{\color{black}#1}}

\definecolor{Red}{RGB}{240,0,0}

\def\udotyh{u_{\bullet \ybar}}
\def\udotyhh{u_{\bullet \ybar'}}

\def\y{\bm{y}}

\def\ybar{\bar{\bm{y}}}

\def\ybarrr{\hat{\bm{y}}} 

\def\zdotyh{z_{\bullet \ybar}}
\def\zdotyhh{z_{\bullet \ybar'}}

\def\Sum{\displaystyle\sum}

\def\B{\mathscr{B}}
\def\By{B_{\ybar}}
\def\F{\mathscr{F}}

\def\Q{\mathscr{Q}}
\def\V{\mathscr{V}}
\def\Qbar{\widetilde{\Q}}

\def\Ydot{Y_{\bullet}}

\def\bR{\mathbb{R}}

\def\p{\partial}

\def\qquad{\quad \quad \quad \quad}

\DeclarePairedDelimiter\abs{\lvert}{\rvert}
\makeatletter
\let\oldabs\abs
\def\abs{\@ifstar{\oldabs}{\oldabs*}}
\newcommand{\normTR}[1]{{\left\vert\kern-0.25ex\left\vert #1 
    \right\vert\kern-0.25ex\right\vert}}
\newcommand{\norme}[1]{{\left\vert\kern-0.25ex\left\vert\kern-0.25ex\left\vert #1 
    \right\vert\kern-0.25ex\right\vert\kern-0.25ex\right\vert}}
\title{Goal-oriented error estimation and adaptivity for stochastic collocation FEM}
\author{Alex Bespalov}
\address{School of Mathematics, University of Birmingham, Edgbaston, Birmingham B15 2TT, UK}
\email{a.bespalov@bham.ac.uk}
\author{Dirk Praetorius}
\address{Institute of Analysis and Scientific Computing, TU Wien, Wiedner Hauptstra\ss{}e~8--10, 1040 Vienna, Austria}
\email{dirk.praetorius@asc.tuwien.ac.at}
\author{Thomas Round}
\address{School of Mathematics, University of Birmingham, Edgbaston, Birmingham B15 2TT, UK}
\email{twr574@outlook.com}
\author{Andrey Savinov}
\address{School of Mathematics, University of Birmingham, Edgbaston, Birmingham B15 2TT, UK}
\email{avs296@bham.ac.uk}


\thanks{{\em Acknowledgements.}
The work of the first author was supported by the EPSRC grant EP/W010925/1.
The work of the second author is supported by the Austrian Science Fund (FWF) under grants
\href{https://doi.org/10.55776/F65}{10.55776/F65} and
\href{https://doi.org/10.55776/P33216}{10.55776/P33216}.
The first and the second authors
would like to thank the Erwin Schr\"odinger International Institute for Mathematics and Physics (ESI)
at the University of Vienna for support and hospitality during the workshops on
\emph{Adaptivity, high dimensionality and randomness} (April 4--8, 2022) and
\emph{Approximation of high-dimensional parametric PDEs in forward UQ} (May 9--13, 2022),
where part of the work on this paper was undertaken.}

\begin{document}

\begin{abstract}
We propose and analyze a general goal-oriented adaptive strategy for approximating quantities of interest (QoIs)
associated with solutions to linear elliptic partial differential equations with random inputs.
The QoIs are represented by bounded linear or continuously G{\^ a}teaux differentiable nonlinear goal functionals, and
the approximations are computed using the sparse grid stochastic collocation finite element method (SC-FEM).
The proposed adaptive strategy relies on novel reliable a posteriori estimates of the errors in approximating QoIs.
One of  the key features of our error estimation approach is the introduction of a correction term into the
approximation of QoIs in order to compensate for the lack of (global) Galerkin orthogonality in the SC-FEM setting.
Computational results generated using the proposed adaptive algorithm are presented in the paper
for representative elliptic problems with affine and nonaffine parametric coefficient dependence and for a range
of linear and nonlinear goal functionals.
\end{abstract}

\maketitle
\thispagestyle{fancy}

\section{Introduction} \label{sec:intro}

Partial differential equations (PDEs) with uncertain inputs 
are ubiquitous in the mathematical modelling of real-life phenomena and in predictive numerical simulations.
Practical simulation scenarios often target certain quantities of interest (QoIs)---specific (and usually localized) features of the solution to a PDE,
e.g., pointwise values of the solution or its fluxes across parts of the computational domain boundary.
In these scenarios, QoIs are represented by functionals of the solution to the underlying PDE with parametric uncertainty.

High-fidelity numerical solutions to parametric PDEs 
are infeasible to compute in many realistic applications.
Conversely, so-called \emph{surrogate approximations} that are functions of the (stochastic) parameters are not only effective
in approximating the input-output map but can also be used to estimate a wide range of QoIs.
However, surrogate approximations introduce numerical errors that need to be accounted for when estimating QoIs.
In this work we address the reliable control of errors in computing QoIs from surrogate approximations of solutions to parameteric PDEs
and systematic approaches to reduce those errors via adaptive algorithms.
%
%
In this context, the aspects of so-called goal-oriented a posteriori error estimation and error control were studied
in~\cite{bryantPrudhommeWildey2015} for generic surrogate approximations of solutions to parametric PDEs,
in~\cite{almeidaOden2010} for stochastic collocation approximations, and
in~\cite{mathelinLeMaitre2007, butlerDawsonWildey2011, bprr18+, bpr2023}
for surrogate approximations obtained via stochastic Galerkin finite element methods. 

\revision{Goal-oriented adaptivity has also been studied in the context of Monte Carlo methods (see, for example \cite{tsgu2013, blst22}), although this class of methods typically only reveals information related to statistical moments of the solution to a given parametric PDE problem.
By contrast, methods which produce surrogate approximations yield explicit information about the solution's dependence on individual random parameters, and furthermore, provide the opportunity to evaluate a larger class of QoIs.}

In this paper, we focus on one particular type of \revision{surrogate approximations---sparse} grid stochastic collocation finite element approximations
that combine sparse grid interpolation in the parameter domain with finite element discretization in the physical (spatial) domain.
In~\cite{almeidaOden2010}, the authors design a goal-oriented adaptive algorithm for the stochastic collocation finite element method (SC-FEM) to approximate \emph{linear} QoIs
derived from the solution to the advection-diffusion problem with uncertain diffusion coefficient, velocity field, \revision{and source term}.
In that work, the errors in the corresponding linear functionals are \emph{approximated} by sampling the finite element error estimates at collocation points,
and the associated error indicators are used to guide the adaptive algorithm.
In the present paper, we aim \revision{to approximate} QoIs represented by \emph{linear and nonlinear} goal functionals
by applying the SC-FEM to the underlying model PDE problem with uncertain inputs.
We build upon the a posteriori error analysis in~\cite{bsx22, BS23}
and use the duality technique in the spirit of~\cite{prudhommeOden1999, gilesSuli2002} to derive,
for the first time to the best of our knowledge, \emph{reliable a posteriori estimates} for the error in approximating QoIs within the SC-FEM framework.
The obtained error estimates are written in terms of spatial and parametric error indicators (for the primal and dual solutions)
that guide refinements of finite element approximations and sparse grids in the proposed goal-oriented adaptive algorithm.

There are two distinctive features in our goal-oriented error estimation approach.
Firstly, we introduce a `correction term' into the approximation of the QoI whose
role 
is to compensate for the lack of (global) Galerkin orthogonality in the SC-FEM setting.
Our numerical experiments show that the correction term has a stabilizing effect on the decay of errors in the goal functional,
so that these errors decay with approximately the same rate as (computable) goal-oriented error estimates.
Secondly, we develop a systematic approach to the error estimation for a large class of goal functionals,
including bounded linear and continuously G{\^ a}teaux differentiable nonlinear functionals.

The paper is organized as follows. 
In section~\ref{sec:model:problem}, we introduce the model problem and recast it in a weak form using
samplewise and combined spatial-parametric formulations.
Section~\ref{sec:sc-fem} recalls the main ideas of the sparse grid stochastic collocation,
its combination with finite element approximations, as well as the a posteriori error estimation strategy for SC-FEM as developed in~\cite{bsx22}.
In section~\ref{sec:goascfem}, we consider a special case of bounded linear goal functionals,
for which we develop a goal-oriented error estimation strategy and design the associated adaptive algorithm.
Then, in section~\ref{sec:goascfem:nl}, we show how this error estimation strategy extends to a larger class of (possibly nonlinear) goal functionals
and propose the corresponding modifications to the previously designed adaptive algorithm.
Computational results generated using the proposed adaptive procedure 
for representative elliptic problems with affine and nonaffine parametric coefficient dependence and for a range
of linear and nonlinear goal functionals are discussed in section~\ref{sec:numerics}.

\section{Problem formulation} \label{sec:model:problem}

Let $D \subset \R^2$ be a bounded Lipschitz domain with polygonal boundary $\partial D$.
Let $\Gamma := \G_1 \times \G_2 \ldots \times \G_M$ denote the parameter domain in $\R^M$,
where $M \in \N$ and each $\G_m$ ($m = 1,\ldots,M$) is a bounded interval in~$\R$.
For simplicity of the presentation, and without loss of generality,
we assume that $\G_1 = \G_2 = \ldots = \G_M = [-1, 1]$.
We introduce a probability measure $\pi(\y) := \prod_{m=1}^M \pi_m(y_m)$ on $(\G,\mathcal{B}(\G))$;
here, $\pi_m$ denotes a Borel probability measure on $\G_m$ ($m = 1,\ldots,M$) and
$\mathcal{B}(\G)$ is the Borel $\sigma$-algebra on $\G$.

We consider the following parametric model problem:
find $u: \overline D \times \Gamma \to \bR$ satisfying
\begin{equation} \label{eq:pde:strong}
\begin{aligned}
-\nabla \cdot (a(x,\y) \nabla u(x,\y)) &= f(x), \quad && x \in D,\\ 
u(x,\y) &= 0, \quad &&x \in \p D 
\end{aligned}
\end{equation}
$\pi$-almost everywhere on $\G$ (i.e., almost surely).
Here, the function $f \in H^{-1}(D)$ represents a deterministic forcing term,
and the coefficient $a(x,\y)$ is a random field on $(\G,\mathcal{B}(\G),\pi)$ over $L^\infty(D)$
that is bounded away from zero and infinity, i.e.,
there exist constants $a_{\min},\, a_{\max}$ such that
\begin{equation} \label{eq:amin:amax}
   0 < a_{\min} \leqslant \underset{x \in D}{\text{ess inf }} a(x,\y) \leqslant
   \underset{x \in D}{\text{ess sup }} a(x,\y) \leqslant a_{\max} < \infty \quad \pi\text{-a.e. on } \Gamma.
\end{equation}

The parametric problem~\eqref{eq:pde:strong} can be understood in a weak sense
by viewing the solution as a map $u: \Gamma \to H_0^1(D) =: \X$.
For $\pi$-almost all $\ybar \in \Gamma$,
this yields the following \emph{samplewise} weak formulation: find $u_{\ybar}(x) := u(x,\ybar) \in \X$ such that
\begin{equation} \label{eq:weak:sampled}
   B_{\ybar}(u_{\ybar}, v) = F(v)
   \quad \forall\,v \in \X, 
\end{equation}
where
\begin{align} \label{eq:sampled:bilinear:form}
   B_{\ybar}(u, v) := \int_D a(x,\ybar) \nabla u(x) \cdot \nabla v(x) \dx,\quad
   F(v) := \int_D f(x) v(x) \dx.
\end{align}
The assumptions on $a$ and $f$ guarantee the well-posedness of~\eqref{eq:weak:sampled} by the Lax--Milgram lemma,
and the parametric problem~\eqref{eq:pde:strong}
admits a unique weak solution $u$ in the Bochner space $L_\pi^p(\G; \X)$ for any $p \in [1, \infty]$;
see~\cite[Lemma~1.1]{BabuskaNT07} for details.
The variational formulations as in~\eqref{eq:weak:sampled} are utilized in the context of sampling methods
(e.g., Monte Carlo or stochastic collocation) where they are discretized using, e.g., the finite element method.
On the other hand, one can exploit the fact that $u: \Gamma \to \X$ is an element
of the Bochner space $L_\pi^2(\G; \X)$ yielding the following \emph{combined spatial-parametric} weak formulation:
find $u \in \V := L_\pi^2(\G; \X)$ such that
\begin{equation} \label{eq:weak:primal}
   \B(u,v) = \F(v) \quad \forall\, v \in \V, 
\end{equation}
where
\[
      \B(u,v) := \int_\Gamma \int_D a(x,\y) \nabla u(x,\y) \cdot \nabla v(x,\y) \dx \dpi(\y), 
      \quad
      \F(v) := \int_\Gamma \int_D f(x) v(x,\y) \dx \dpi(\y). 
\]
We note that the bilinear form $\B$ induces the energy norm defined by $\enorm{v}{}{} := \big(\B(v,v)\big)^{1/2}$, and
the following norm equivalence holds due to assumption~\eqref{eq:amin:amax}:
\begin{align} \label{eq:norm:equiv}
   a_{\min}^{1/2} \, \norm{v}{} \le \enorm{v}{}{} \le a_{\max}^{1/2} \, \norm{v}{}\quad
   \forall\, v \in \V, 
\end{align}
where $\norm{\cdot}{}$ denotes the standard norm in the Bochner space $\V = L_\pi^2(\G; \X)$.
The well-posedness of~\eqref{eq:weak:primal} follows again by the Lax--Milgram lemma.
The variational formulation in~\eqref{eq:weak:primal} is the starting point for approximating
the parametric solution to~\eqref{eq:pde:strong} using the stochastic Galerkin FEM.

In this work, rather than approximating the solution $u \in \V$,
we aim at the numerical approximation of the functional value $\Q(u)$,
where $\Q: \V \to \mathbb{R}$ is a continuous goal func\-ti\-on\-al.
Specifically, given a continuous functional $Q: \X \to \R$, we introduce the quantity of interest
$Q(u(\cdot,\y))$ for $\pi$-almost all $\y \in \Gamma$ and consider the goal functional $\Q$ defined by
\begin{equation} \label{eq:goal}
   \Q(v) := \int_\Gamma Q(v(\cdot,\y)) \, \dpi(\y)\quad
   \text{for all $v \in \V$}.
\end{equation}
Thus, $\Q(u)$ is the mean value of the quantity of interest $Q(u)$ derived from the solution $u$ to problem~\eqref{eq:pde:strong}.
In practice, $Q(u)$ can represent a range of linear and nonlinear quantities of interest, for example,
the average of the solution, or the average of the associated convection term $(u,u) \cdot \nabla u$,
across a spatial subdomain, or a pointwise estimate.
In particular, through nonlinear $Q$, the goal functional $\Q$ can represent second and further moments
of such quantities of interest as the average of the solution, or its flux, over a subdomain.

\section{Stochastic collocation finite element method} \label{sec:sc-fem}

Our goal-oriented adaptive algorithm for approximating $\Q(u)$ will employ
the sparse grid stochastic collocation finite element method. 
In this section, we recall the main ideas of SC-FEM, including the construction of the underlying approximation spaces
and the associated a posteriori error analysis.
For the latter, we follow the approach (and the notation) presented in~\cite{bsx22, BS23}.

\subsection{Spatial discretization and mesh refinement} \label{sec:spatial:discrete}

Let $\TT_\bullet$ be a mesh on the spatial domain $D$ (i.e., a conforming triangulation of $D$ into
compact non-degenerate triangles $T$), and let $\NN_\bullet$ denote the set of vertices of $\TT_\bullet$.
For the numerical solution of~\eqref{eq:weak:sampled},
we employ the space $\X_{\bullet}$ of continuous piecewise linear functions,
\begin{equation*}
 \X_\bullet := \SS^1_0(\TT_\bullet) :=
 \{v \in \X : v \vert_T \text{ is affine for all } T \in \TT_\bullet \} \subset \X = H^1_0(D).
\end{equation*}
The standard basis of $\X_\bullet$ is given by
$\{ \varphi_{\bullet,\xi} : \xi \in \NN_\bullet \setminus \partial D \}$, where
$\varphi_{\bullet,\xi}$ denotes the hat function associated with the vertex $\xi \in \NN_\bullet$.

For mesh refinement, we employ newest vertex bisection (NVB); see, e.g., \cite{stevenson,kpp}.
We assume that any mesh $\TT_\bullet$ employed for the spatial discretization
is obtained by (uniform or local) NVB refinement(s) of a given (coarse) initial mesh $\TT_0$.
The finite element space associated with $\TT_0$ is denoted by $\X_0 := \SS^1_0(\TT_0)$.

Let $\widehat\TT_\bullet$ be the mesh obtained by uniform refinement of $\TT_\bullet$
(i.e., all elements of $\TT_\bullet$ are refined by three bisections).
Then, $\widehat\NN_\bullet$ denotes the set of vertices of $\widehat\TT_\bullet$,
and $\NN_\bullet^+ := (\widehat\NN_\bullet \setminus \NN_\bullet) \setminus \partial D$
is the set of new interior vertices created by this refinement of $\TT_\bullet$.
The finite element space associated with $\widehat\TT_\bullet$ is denoted as
$\widehat\X_\bullet := \SS^1_0(\widehat\TT_\bullet)$,
and
$\{ \widehat\varphi_{\bullet,\xi} : \xi \in \widehat\NN_\bullet \setminus \partial D \}$ is the corresponding  basis of hat functions.
Let  $\Y_{\bullet}$ be the approximation space associated with the set $\NN_\bullet^+$ of newly introduced nodes (edge midpoints), i.e.,
$
   \Y_\bullet := \hull\{ \widehat\varphi_{\bullet,\xi} : \xi \in \NN_\bullet^+ \}.
$
Then, the following decomposition holds:
$\widehat \X_{\bullet} = \X_{\bullet} \oplus \Y_{\bullet}$.
\revision{For a set of marked nodes $\MM_{\bullet} \subseteq \NN_{\bullet}^+$,
we introduce a routine $\refine(\TT_{\bullet},\MM_{\bullet})$
that generates a mesh $\TT_\circ$---the coarsest NVB refinement of $\TT_{\bullet}$
such that $\MM_{\bullet} \subset \NN_{\circ}$, i.e., all marked nodes are vertices of~$\TT_\circ$.}

For a fixed point $\ybar \in\G$,
we denote by $u_{\bullet \ybar} \in \X_{\bullet}$ the Galerkin finite element approximation satisfying
\begin{align} \label{eq:sample:fem}
   B_{\ybar}(u_{\bullet \ybar}, v) = F(v)
   \quad \forall v \in \X_{\bullet}.
\end{align} 
The enhanced Galerkin approximation satisfying~\eqref{eq:sample:fem}
for all $v \in \widehat\X_{\bullet}$ is denoted by $\widehat u_{\bullet \ybar} \in \widehat\X_{\bullet}$.

\subsection{Sparse grid collocation and parametric enrichment} \label{sec:sparse:grids}

In the context of the numerical solution of high-dimensional parametric problems,
the state-of-the-art stochastic collocation methods employ the nodes of \emph{sparse grids} as collocation points
$\ybar$ in the parameter domain $\G$.

In the sparse grid SC-FEM, the parametric approximation is associated with
a monotone (or, downward-closed) finite set $\indset_\bullet \subset \N^M$ of multi-indices;
\revision{specifically, 
$\indset_\bullet \,{=}\, \{ \nnu \,{=}\, (\nu_1,\ldots,\nu_M) : \nu_m \in \N,\, m = 1,\ldots,M \}$ is such 
that $\#\indset_\bullet < \infty$ and
\[
   \nnu \in \indset_\bullet \Longrightarrow \nnu - \eeps_m \in \indset_\bullet \quad
   \forall\,m=1,\ldots,M \text{ such that } \nu_m >1,
\]
where $(\eeps_m)_i = \delta_{mi}$ for all $i=1,\ldots,M$.}
Each component $\nu_m$ ($m = 1,\ldots,M$) of the multi-index $\nnu \in \indset_\bullet$ corresponds to
a set $\Colpts_m^{\mf(\nu_m)}$ of $\mf(\nu_m)$ points along the $m$th coordinate axis in $\R^M$,
where $\mf : \N_0 \to \N_0$ is a strictly increasing function satisfying $\mf(0) = 0$, $\mf(1) = 1$.
Crucially, \revision{for each $m=1,2,\dots, M$, the} sets $\Colpts_m^{\mf(\nu_m)}$ come from the same family of \emph{nested sets} 
of 1D nodes on $[-1,1]$;
examples of such node sets include Leja points and Clenshaw--Curtis quadrature points
(one has $\mf(i) = i$ for Leja points and $\mf(i) = 2^{i-1}+1$, $i > 1$ for Clenshaw--Curtis nodes with the doubling rule).
Then, the associated sparse grid $\Colpts_\bullet = \Colpts_{\indset_\bullet}$
of collocation points on $\G$ is  given by
\[
   \Colpts_{\indset_\bullet} := \bigcup_{\nnu \in \indset_\bullet} \Colpts^{\,(\nnu)}
   = \bigcup_{\nnu \in \indset_\bullet}\,
      \Colpts_1^{\mf(\nu_1)} \times \Colpts_2^{\mf(\nu_2)} \times \ldots \times \Colpts_M^{\mf(\nu_M)}.
\]
Given a sparse grid $\Colpts_\bullet$ of collocation points in $\G$,
the SC-FEM approximation of the solution $u$ to parametric problem~\eqref{eq:pde:strong} is constructed as
the Lagrange interpolant:
\begin{equation} \label{eq:scfem:sol}
   \scsol(x,\y) := \sum\limits_{\ybar \in \Colpts_\bullet} u_{\bullet \ybar}(x) \LagrBasis{\bullet \ybar}{}(\y),
\end{equation}
where $\{\LagrBasis{\bullet \ybar}{}(\y), \ybar \in \Colpts_\bullet \}$ is a set of
multivariable Lagrange basis functions constructed for the set of collocation points $\Colpts_\bullet$ and satisfying
$\LagrBasis{\bullet \ybar}{}(\ybar') = \delta_{\ybar\ybar'}$ for any $\ybar,\, \ybar' \in \Colpts_\bullet$.
\revision{The total number of degrees of freedom in the SC-FEM approximation $\scsol$ defined by~\eqref{eq:scfem:sol} is given by
$\big(\dim(\X_{\bullet})\big) \times \big(\#\Colpts_\bullet\big)$.}

The enhancement of the parametric component of the SC-FEM approximation given by~\eqref{eq:scfem:sol}
is effected by enriching the index set~$\indset_\bullet$. 
To that end, 
we first define the \emph{margin} of $\indset_\bullet$ as
\[
   \marg({\indset_\bullet}) :=
   \big\{ \nnu \in \N^M \setminus \indset_\bullet :
   \nnu - \eeps_m \in \indset_\bullet \text{ for some } m \in \{1, \dots, M \} \big\}.
\] 
We enrich the index set $\indset_\bullet$ by adding multi-indices from a subset of $\marg(\indset_{\bullet})$
that is called the \emph{reduced margin}:
\[
   \rmarg({\indset_\bullet}) :=
   \big\{ \nnu \in \marg(\indset_\bullet) :
   \nnu - \eeps_m \in \indset_\bullet \text{ for all } m = 1, \dots, M \text{ such that } \nu_m > 1 \big\}.
\]
Any such enrichment of $\indset_\bullet$ 
corresponds to adding some collocation points from the set $\widehat\Colpts_\bullet \setminus \Colpts_\bullet$,
where $\widehat\Colpts_\bullet := \Colpts_{\indset_\bullet \cup \rmarg(\indset_\bullet)}$. 
\revision{It is important to note that if $\indset_\bullet$ is monotone,
then for any subset $\Upsilon \subseteq \rmarg(\indset_\bullet)$, the index
set $\indset_\bullet \cup \Upsilon$ is also monotone.}

Note that the SC-FEM solution considered in the present work
follows the so-called \emph{single-level} construction that employs the same finite element space $\X_{\bullet}$
for all collocation points $\ybar \in \Colpts_\bullet$ (cf.~\cite{BabuskaNT07,NobileTW08a,GuignardN18,bsx22}).
This is in contrast to \emph{multilevel} SC-FEM approximations that allow
$\X_{\bullet \ybar} \not= \X_{\bullet \ybar'}$ for $\ybar \not= \ybar'$; see, e.g.,~\cite{LangSS20,FeischlS21,BS23}.
Our choice of the single-level SC-FEM in this paper is motivated by the results of numerical experiments in~\cite{bsx22, BS23}.
These results have indicated that the single-level version is likely
to be more efficient when the same adaptively refined finite element mesh can adequately resolve solution features
for a range of individually sampled problems, which is often the case for the model parametric problem~\eqref{eq:pde:strong}.

\subsection{A posteriori error estimation} \label{sec:error:estimation}

In this section, we briefly recall the a posteriori error estimation strategy developed in~\cite{bsx22}
as well as the associated error indicators that steer the adaptive refinement of SC-FEM approximations;
we refer to~\cite[section~4]{bsx22} for full details.
We recall that $\norm{\cdot}{}$ denotes the norm in $\V = L^2_{\pi}(\G,\X)$ with $\X = H_0^1(D)$;
we also define $\norm{ {\cdot} }{\X} := \norm{\nabla\cdot}{L^2(D)}$.

A key ingredient of the error estimation strategy developed in~\cite{bsx22} is an
\emph{enriched} SC-FEM approximation, denoted by $\scsolhat$, that is generated using spatial and parametric
enhancements of the SC-FEM approximation $\scsol$.
The \emph{spatial} enhancement of 
$\scsol$ is performed by uniform refinement of the underlying mesh,
whereas the \emph{parametric} enhancement 
is done by adding the full reduced margin $\rmarg({\indset_\bullet})$ to the current index set $\indset_\bullet$.
The use of these enhanced approximations allows one to independently control the spatial and parametric 
contributions to the overall discretization error $u - \scsol$.
%
In particular, under the saturation assumption that $\scsolhat$ reduces the discretization error, 
i.e.,
\begin{equation} \label{eq:saturation}
   \norm{u - \scsolhat}{} \le q_{\rm sat} \, \norm{u - \scsol}{}
\end{equation}
with a saturation constant $q_{\rm sat} \in (0,1)$ independent of the discretization parameters,
the error estimate in the SC-FEM approximation $\scsol$
is given by the sum of spatial and parametric contributions as follows (see equations~(22)--(24) 
and~(33) in~\cite{bsx22}):
\begin{equation} \label{eq:error:estimate}
   \norm{u - \scsol}{} \le \big(1 - q_{\rm sat}\big)^{-1} \, \norm{\scsolhat - \scsol}{} \le
   \big(1 - q_{\rm sat}\big)^{-1} \, \big(\mu_\bullet + \tau_\bullet\big).
\end{equation}
Here, $\mu_\bullet = \mu_\bullet[u]$ and $\tau_\bullet = \tau_\bullet[u]$ are, respectively, the spatial and parametric error estimates defined~by
\begin{align} \label{eq:spatial:estimate}
   \mu_\bullet = \mu_\bullet[u] & := 
    \norm[\bigg]{\sum\limits_{\ybar \in \Colpts_\bullet} (\widehat u_{\bullet \ybar} - u_{\bullet \ybar})\, \LagrBasis{\bullet \ybar}{}}{}
\end{align}
and
\begin{align} \label{eq:param:estimate}
   \tau_\bullet = \tau_\bullet[u]  & : =
   \norm[\bigg]
   {\sum\limits_{\ybar' \in \widehat\Colpts_\bullet \setminus \Colpts_\bullet}
                         \big( u_{\bullet \ybar'} - \scsol(\cdot,\ybar') \big) 
                         \LagrBasisHat{\bullet \ybar'}{}}{},
\end{align}
where
$\LagrBasisHat{\bullet \ybar'}{}(\y)$ 
denotes the multivariable Lagrange basis function associated with the point $\ybar' \in \widehat\Colpts_\bullet$ and satisfying
$\LagrBasisHat{\bullet \ybar'}{}(\ybar'') = \delta_{\ybar'\ybar''}$ for any $\ybar',\,\ybar'' \in \widehat\Colpts_\bullet$.

\revision{
\begin{remark} \label{remark:saturation}
While the saturation assumption can be empirically justified when numerical approximations exhibit some asymptotic convergence behavior,
rigorous theoretical proofs exist either for deterministic problems with constant coefficients
(see, e.g.,~\cite{dn2002,cgg2016}) or in the context of stochastic Galerkin FEM for parametric PDEs (see~\cite{BachmayrEEV_CAF}). In the context of SC-FEM, however,
the proof of~\eqref{eq:saturation} remains an open problem.
\end{remark}
}

Let us now turn to local error indicators that are associated with
the error estimates in~\eqref{eq:spatial:estimate} and~\eqref{eq:param:estimate}.
For each $\ybar \in \Colpts_\bullet$, we define the \emph{spatial} two-level error indicators
associated with interior edge midpoints
(recall that $\Y_\bullet := \hull\{ \widehat\varphi_{\bullet,\xi} : \xi \in \NN_\bullet^+ \}$):
\begin{equation} \label{eq:2level:local:indicator}
   \mu_{\bullet \ybar}(\xi) :=
   \frac{\big| F(\widehat\varphi_{\bullet,\xi}) - B_{\ybar}(u_{\bullet \ybar}, \widehat\varphi_{\bullet,\xi}) \big|}
           {\norm{\widehat\varphi_{\bullet,\xi}}{\X}}\quad
   \text{for all } \xi \in \NN_{\bullet}^+.
\end{equation}
We combine these \emph{local} indicators to define the \emph{spatial} error indicator for each $\ybar \in \Colpts_\bullet$:
\begin{equation} \label{eq:2level:indicator}
   \mu^2_{\bullet \ybar} := \sum\limits_{\xi \in \NN_{\bullet}^+} \mu^2_{\bullet \ybar}(\xi).
\end{equation}
Next, for each $\nnu \in \rmarg({\indset_\bullet})$, the \emph{parametric} error indicator is defined as follows:
\begin{align} \label{eq:param:indicator}
   \widetilde\tau_{\bullet \nnu} =
   \widetilde\tau_{\bullet \nnu}[u] :=
   \sum\limits_{\ybar' \in \widetilde\Colpts_{\bullet \nnu}}
   \norm[\big]{u_{\bullet \ybar'} - \scsol(\cdot,\ybar')}{\X} \,
   \norm{\LagrBasisHat{\bullet \ybar'}{}}{L_\pi^{2}(\G)},
\end{align}
where
$\widetilde\Colpts_{\bullet \nnu} := \Colpts_{\indset_\bullet \cup \{\nnu\}}  \setminus \Colpts_{\indset_\bullet}
         \subset \widehat\Colpts_{\bullet} \setminus \Colpts_{\bullet}$
are the collocation points `generated' by the multi-index $\nnu \in  \rmarg({\indset_\bullet})$.

We refer to~\cite[section~3]{BS23} for a discussion of computational costs
associated with computing the error estimates $\mu_\bullet$ and $\tau_\bullet$ given by~\eqref{eq:spatial:estimate}
and~\eqref{eq:param:estimate}.
The key point is that in the adaptive algorithm, the computation of these error estimates is only required to give
a reliable criterion for termination of the adaptive process.
In contrast  to the error estimates $\mu_\bullet$ and $\tau_\bullet$,
the error indicators $\mu_{\bullet\ybar}$ and $\widetilde\tau_{\bullet\nnu}$ are cheaper to compute
and the following inequalities hold (see equation~(31) and 
Remark~3 in~\cite{bsx22}):
\begin{equation} \label{eq:err:indicators}
   \mu_\bullet \lesssim
   \bar\mu_\bullet := \sum\limits_{\ybar \in \Colpts_\bullet} \mu_{\bullet\ybar}\, \|L_{\bullet\ybar}\|_{L^2_\pi(\G)}
   \quad\text{and}\quad
   \tau_\bullet \le
   \bar\tau_\bullet := \sum\limits_{\nnu \in \rmarg(\indset_\bullet)} \widetilde\tau_{\bullet\nnu}.
\end{equation}
This motivates the use of 
the \emph{error indicators}~\eqref{eq:2level:local:indicator} and~\eqref{eq:param:indicator} in the marking strategy within the adaptive algorithm.

\section{Goal-oriented adaptive SC-FEM with linear goal functionals} \label{sec:goascfem}

In this section, we assume that the quantity of interest $Q(u(\cdot,\y))$ is obtained using a bounded linear functional $Q: \X \to \R$,
i.e., $Q \in \X^*$.
Consequently, the goal functional $\Q: \V \to \R$ defined in~\eqref{eq:goal} is also linear and bounded, i.e., $\Q \in \V^*$.
Focusing on this special case, we develop a goal-oriented error estimation strategy in the context of SC-FEM
and present the associated adaptive algorithm.
Subsequently, in section~\ref{sec:goascfem:nl}, we will extend this error estimation strategy and adaptive algorithm
to a wider class of (possibly nonlinear) goal functionals.

\subsection{Dual formulations and the goal-oriented error estimate} \label{sec:goal:error:estimation}

We use the standard approach to goal-oriented error estimation (see, e.g.,~\cite{MR1352472, beckerRannacher2001, gilesSuli2002})
but adapt it to the setting of SC-FEM.
To that end, we introduce two \emph{dual formulations}: one associated with the functional $Q \in \X^* = H^{-1}(D)$
and the other associated with the functional $\Q \in \V^*$ given by~\eqref{eq:goal}.
For $\pi$-almost all $\ybar \in \Gamma$, the first dual formulation corresponds to the
\emph{samplewise} primal formulation~\eqref{eq:weak:sampled} and reads as:
find $z_{\ybar}(x) := z(x,\ybar) \in \X$ such that
\begin{equation} \label{eq:dual:sampled}
   B_{\ybar}(v,z_{\ybar}) = Q(v)
   \quad \forall\,v \in \X. 
\end{equation}
The second dual formulation corresponds to the \emph{combined} primal formulation~\eqref{eq:weak:primal}
 and reads as:
 find $z \in \V$ satisfying
\begin{equation} \label{eq:weak:dual}
   \B(v,z) = \Q(v) \quad \forall v \in \V.
\end{equation}
Both dual formulations are well-posed due to the Lax--Milgram lemma.

In the same way as we did for the primal formulation,
we discretize~\eqref{eq:dual:sampled} at each collocation point $\ybar \in \Colpts_\bullet$ by the Galerkin FEM and
build the Lagrange interpolant out of the resulting Galerkin approximations.
Specifically, for each $\ybar \in \Colpts_\bullet$, we denote by $z_{\bullet \ybar} \in \X_{\bullet}$ the Galerkin finite element approximation satisfying
\begin{align} \label{eq:sample:fem:dual}
   B_{\ybar}(v,z_{\bullet \ybar}) = Q(v)
   \quad \forall v \in \X_{\bullet}
\end{align}
and define
\begin{equation} \label{eq:scfem:sol:dual}
   \scsoldual(x,\y) := \sum\limits_{\ybar \in \Colpts_\bullet} z_{\bullet \ybar}(x) \LagrBasis{\bullet \ybar}{}(\y).
\end{equation}
To estimate the error $\norm{z - \scsoldual}{}$, we use the error estimation strategy described in~\S\ref{sec:error:estimation}.
In particular, we define the error estimates
\begin{align} \label{eq:dual:estimates}
   \eta_\bullet := \mu_\bullet[z],\quad
   \sigma_\bullet := \tau_\bullet[z]
\end{align}
and the parametric error indicator
\begin{align} \label{eq:param:indicator:dual}
   \widetilde\sigma_{\bullet \nnu} := \widetilde\tau_{\bullet \nnu}[z]\quad
   \text{ for each $\nnu \in \rmarg({\indset_\bullet})$},
\end{align}
where $\mu_\bullet[\cdot]$, $\tau_\bullet[\cdot]$, and $\widetilde\tau_{\bullet \nnu}[\cdot]$ are given by~\eqref{eq:spatial:estimate},
\eqref{eq:param:estimate}, and \eqref{eq:param:indicator}, respectively.
Also, similarly to~\eqref{eq:2level:local:indicator} and~\eqref{eq:2level:indicator},
for each $\ybar \in \Colpts_\bullet$, we define the spatial two-level error indicators
\begin{equation} \label{eq:2level:local:indicator:dual}
   \eta_{\bullet \ybar}(\xi) :=
   \frac{\big| Q(\widehat\varphi_{\bullet,\xi}) - B_{\ybar}(\widehat\varphi_{\bullet,\xi}, z_{\bullet \ybar}) \big|}
           {\norm{\widehat\varphi_{\bullet,\xi}}{\X}}\quad
   \text{for all } \xi \in \NN_{\bullet}^+
\end{equation}
and
\begin{equation} \label{eq:2level:indicator:dual}
   \eta^2_{\bullet \ybar} := \sum\limits_{\xi \in \NN_{\bullet}^+} \eta^2_{\bullet \ybar}(\xi).
\end{equation}
Similarly to~\eqref{eq:error:estimate} and \eqref{eq:err:indicators}, the following estimates hold:
\begin{equation} \label{eq:error:estimate:dual}
   \norm{z - \scsoldual}{} \le \big(1 - q_{\rm sat}\big)^{-1} \, \big(\eta_\bullet + \sigma_\bullet\big)
\end{equation}
as well as
\begin{equation} \label{eq:err:indicators:dual}
   \eta_\bullet \lesssim
   \bar\eta_\bullet := \sum\limits_{\ybar \in \Colpts_\bullet} \eta_{\bullet\ybar}\, \|L_{\bullet\ybar}\|_{L^2_\pi(\G)}
   \quad\text{and}\quad
   \sigma_\bullet \le
   \bar\sigma_\bullet := \sum\limits_{\nnu \in \rmarg(\indset_\bullet)} \widetilde\sigma_{\bullet\nnu}.
\end{equation}

Now, recall that our aim is to approximate $\Q(u)$, where
$\Q \in \V^*$ is the goal functional given by~\eqref{eq:goal} and
$u$ is the solution to the parametric problem~\eqref{eq:pde:strong}.
Exploiting the linearity of $\Q$ and using~\eqref{eq:weak:dual}, we have
\begin{align} \label{eq:goal:error:1}
   \Q(u) - \Q(\scsol) &= \Q(u - \scsol) = \B(u - \scsol, z)
   \nonumber
   \\[3pt]
   &= \B(u - \scsol, z - \scsoldual)  + \B(u - \scsol, \scsoldual).
\end{align}
At this point, the standard approach (e.g., for Galerkin FEM approximations) would exploit the Galerkin orthogonality property
and the Cauchy--Schwartz inequality to derive the upper bound for the error in the goal functions 
as the product of energy errors, i.e., $\enorm{u - \scsol}{}{}\, \enorm{z - \scsoldual}{}{}$.
In the setting of stochastic collocation FEM, $\B(u - \scsol, \scsoldual) \not= 0$.
On the other hand, this quantity is computable:
\[
   \B(u - \scsol, \scsoldual) = \B(u,\scsoldual) - \B(\scsol,\scsoldual) \reff{eq:weak:primal}{=}
   \F(\scsoldual) - \B(\scsol,\scsoldual).
\]
Therefore, we can use this quantity as a `correction term'
to define the approximation of $\Q(u)$ as follows:
\begin{align} \label{eq:goal:approx}
   \Q(u) \,{\approx}\, \Qbar(\scsol,\scsoldual) \,{:=}\, \Q(\scsol) + \B(u - \scsol, \scsoldual) \,{=}\,
   \Q(\scsol) + \F(\scsoldual) - \B(\scsol,\scsoldual).
\end{align}
Hence, applying the Cauchy--Schwartz inequality, we obtain from~\eqref{eq:goal:error:1}:
\begin{align*} 
   \big|\Q(u) - \Qbar(\scsol,\scsoldual)\big| = \big| \B(u - \scsol, z - \scsoldual) \big| \le
   \enorm[\big]{u - \scsol}{}{}\, \enorm[\big]{z - \scsoldual}{}{}.
\end{align*}
Now, using the norm equivalence~\eqref{eq:norm:equiv} and the a posteriori error estimates~\eqref{eq:error:estimate}
and~\eqref{eq:error:estimate:dual} for the primal and dual SC-FEM approximations, we derive the final error estimate
for the approximation of $\Q(u)$ by $\Qbar(\scsol,\scsoldual)$:
\begin{align} \label{eq:goal:error:estimate}
   \big|\Q(u) - \Qbar(\scsol,\scsoldual)\big| \lesssim
   \norm[\big]{u - \scsol}{}\, \norm[\big]{z - \scsoldual}{} \lesssim
   \big(\mu_\bullet + \tau_\bullet\big) \, \big(\eta_\bullet + \sigma_\bullet\big),
\end{align}
where 
$\mu_\bullet = \mu_\bullet[u]$, $\tau_\bullet = \tau_\bullet[u]$,
$\eta_\bullet = \mu_\bullet[z]$, $\sigma_\bullet = \tau_\bullet[z]$,
with $\mu_\bullet[\cdot]$ and $\tau_\bullet[\cdot]$ given by~\eqref{eq:spatial:estimate} and
\eqref{eq:param:estimate}, respectively, and
the hidden constant only depends on $a_{\max}$ and the saturation constants $q_{\rm sat}$ (for the primal and dual solutions).

We emphasize that in contrast to the standard goal-oriented approach in the context of Galerkin methods,
where $\Q(u)$ is typically approximated by evaluating the goal functional on the \emph{primal} Galerkin solution $u_\bullet$,
i.e., $\Q(u) \approx \Q(u_\bullet)$, the approximation of $\Q(u)$ in the SC-FEM setting is best done using
$\Qbar(\scsol,\scsoldual)$, which depends on both \emph{primal and dual} SC-FEM approximations; cf.~\eqref{eq:goal:approx}.
We also note that, as iterations progress, it is expected that $\scsol$ converges to $u$ and, therefore,
the `correction term' given by $\B(u - \scsol, \scsoldual)$ converges to zero.

\subsection{Adaptive algorithm} \label{sec:algorithm}

In this section, we present a goal-oriented adaptive algorithm that employs SC-FEM approximations
of the primal and dual solutions satisfying~\eqref{eq:weak:primal} and~\eqref{eq:weak:dual}, respectively. \revision{It originates from the adaptive sparse grid algorithm proposed in~\cite{GerstnerGriebel2003}.}
The refinement of the underlying finite element mesh and the enrichment of the set of collocation points
are steered by, respectively, the spatial error indicators $\mu_{\bullet \ybar}$, $\eta_{\bullet \ybar}$ and
the parametric indicators $\widetilde\sigma_{\bullet \nnu}$, $\widetilde\tau_{\bullet \nnu}$
discussed in~\S\ref{sec:error:estimation} and~\S\ref{sec:goal:error:estimation}.
The calculation of the error estimates $\mu_{\bullet}$, $\tau_{\bullet}$, $\eta_{\bullet}$, $\sigma_{\bullet}$
in~\eqref{eq:spatial:estimate}, \eqref{eq:param:estimate}, \eqref{eq:dual:estimates} needs to be done periodically,
since the product error estimate in~\eqref{eq:goal:error:estimate} is only required for reliable termination of the adaptive process
according to the given tolerance.
In practice, the approximation $\Qbar(\scsol,\scsoldual)$ of $\Q(u)$ only needs to be computed once,
when the error tolerance is reached (see step~(vii) in Algorithm~\ref{algorithm} below).
However, in the numerical experiments presented in section~\ref{sec:numerics},
we compute $\Qbar(\scsol,\scsoldual)$ at each iteration to illustrate the decay of the error in approximating~$\Q(u)$.

\begin{algorithm} \label{algorithm}
{\bfseries Input:}
$\indset_0 = \{\1\}$; $\TT_0$; marking criterion.\\
Set the iteration counter $\ell := 0$, the output counter $k$, and the error tolerance {\tt tol}.
\begin{itemize}
\item[\rm(i)] 
Compute the primal and dual Galerkin approximations
$\big\{ u_{\ell \ybar} \in \X_{\ell} : \ybar \in \Colpts_{\indset_\ell \cup \rmarg(\indset_\ell)} \big\}$ and
$\big\{ z_{\ell \ybar} \in \X_{\ell} : \ybar \in \Colpts_{\indset_\ell \cup \rmarg(\indset_\ell)} \big\}$
by solving~\eqref{eq:sample:fem} and~\eqref{eq:sample:fem:dual}, respectively.
\item[\rm(ii)] 
Compute local spatial error indicators
$\big\{ \mu_{\ell \ybar}(\xi) : \ybar \in \Colpts_\ell,\ \xi \in \NN_{\ell}^{+} \big\}$ and
$\big\{ \eta_{\ell \ybar}(\xi) : \ybar \in \Colpts_\ell,\ \xi \in \NN_{\ell}^{+} \big\}$
given by~\eqref{eq:2level:local:indicator} and~\eqref{eq:2level:local:indicator:dual}, respectively,
as well as their global counterparts
$\big\{ \mu_{\ell\ybar} : \ybar \in \Colpts_\ell \big\}$ and
$\big\{ \eta_{\ell\ybar} : \ybar \in \Colpts_\ell \big\}$
given by~\eqref{eq:2level:indicator} and~\eqref{eq:2level:indicator:dual}, respectively.
\item[\rm(iii)] 
Compute parametric error indicators
$\big\{ \widetilde\tau_{\ell \nnu} : \nnu \in \rmarg(\indset_\ell) \big\}$ and
$\big\{ \widetilde\sigma_{\ell \nnu} : \nnu \in \rmarg(\indset_\ell) \big\}$
given by~\eqref{eq:param:indicator} and~\eqref{eq:param:indicator:dual}, respectively.
\item[\rm(iv)] 
Use a marking criterion (e.g., Algorithm~\ref{alg:marking} below) to determine
$\MM_{\ell} \subseteq \NN_{\ell}^+$ and
$\markindset_\ell \subseteq \rmarg(\indset_\ell)$.
\item[\rm(v)] Set $\TT_{\ell+1} := \refine(\TT_{\ell},\MM_{\ell})$.
\item[\rm(vi)] Set $\indset_{\ell+1} := \indset_\ell \cup \markindset_\ell$,
$\Colpts_{\ell+1} := \Colpts_{\indset_{\ell+1}}$.
\item[\rm(vii)] 
If $\ell = j k,\ j \in \N$, compute the 
error estimates
$\mu_\ell$, $\eta_\ell$, $\tau_\ell$, $\sigma_\ell$ given by~\eqref{eq:spatial:estimate},~\eqref{eq:param:estimate}, and~\eqref{eq:dual:estimates}.
If $(\mu_\ell + \tau_\ell) \, (\eta_\ell + \sigma_\ell)< {\tt tol}$, then
compute $\Qbar(u_{\ell}^{\rm SC},z_{\ell}^{\rm SC})$ using~\eqref{eq:goal:approx} and exit;
here, the SC-FEM approximations $u_{\ell}^{\rm SC}$ and $z_{\ell}^{\rm SC}$ are
computed via~\eqref{eq:scfem:sol} and~\eqref{eq:scfem:sol:dual}, respectively, from the Galerkin approximations
$\big\{ u_{{\ell} \ybar},\, z_{{\ell} \ybar} \in \X_{\ell} : \ybar \in \Colpts_{\ell} \big\}$.
\item[\rm(viii)] Increase the counter $\ell \mapsto \ell+1$ and goto {\rm(i)}.
\end{itemize}
{\bfseries Output:} For some specific  $\ell_*=jk \in \N$,
the algorithm returns the approximation $\Qbar(u_{\ell_*}^{\rm SC},z_{\ell_*}^{\rm SC})$ of the goal functional $\Q(u)$
together with a corresponding error estimate $(\mu_{\ell_*} + \tau_{\ell_*}) \, (\eta_{\ell_*} + \sigma_{\ell_*})$.
\end{algorithm}

A marking strategy needs to be specified for step~(iv) of Algorithm~\ref{algorithm}.
The marking strategy serves two purposes.
Firstly, it identifies the refinement type (spatial versus parametric) and secondly,
it generates the corresponding marking sets.
In the context of goal-oriented adaptivity, both these tasks are complicated by the fact that there are two sets of spatial indicators
and two sets of parametric indicators, with each set corresponding to either primal or dual approximations.
To decide on the refinement type at each iteration, we use the \emph{cumulative error indicators}
$\bar\mu_\ell,\, \bar\eta_\ell,\, \bar\tau_\ell,\, \bar\sigma_\ell$
defined in~\eqref{eq:err:indicators} and~\eqref{eq:err:indicators:dual}
(recall that these error indicators are cheaper to compute than the error estimates $\mu_\ell,\, \eta_\ell,\, \tau_\ell,\, \sigma_\ell$).
The refinement type is determined by
the following quantity:
\begin{equation} \label{eq:identify:ref:type}
   \max\big\{ \bar\mu_\ell^2 + \bar\eta_\ell^2,\, \bar\tau_\ell^2 + \bar\sigma_\ell^2 \big\}.
\end{equation}
For example, if $\bar\mu_\ell^2 + \bar\eta_\ell^2 \ge \bar\tau_\ell^2 + \bar\sigma_\ell^2$,
then the algorithm performs \emph{spatial} refinement.
This approach to choosing the refinement type is motivated by the following bound on the product estimate
for the error in the goal functional (cf.~\eqref{eq:goal:error:estimate}):
\begin{align} \label{eq:ref:type:motiv}
   \big(\mu_\ell + \tau_\ell\big) \big(\eta_\ell + \sigma_\ell\big)
   &\le
   \big(\bar\mu_\ell + \bar\tau_\ell\big) \big(\bar\eta_\ell + \bar\sigma_\ell\big)
   \nonumber
   \\[3pt]
   &\le
   \frac 12
   \Big[
           \big(\bar\mu_\ell + \bar\tau_\ell\big)^2 + \big(\bar\eta_\ell + \bar\sigma_\ell\big)^2
   \Big]
   \le
   \big(\bar\mu_\ell^2 + \bar\eta_\ell^2\big) + \big(\bar\tau_\ell^2 + \bar\sigma_\ell^2\big).
\end{align}
Thus, the proposed approach identifies a dominant source of errors and selects the corresponding refinement type.
There are alternative ways to determine the refinement type from the above four cumulative error indicators.
For example, one can use
$\max\{ \bar\mu_\ell,\, \bar\tau_\ell,\, \bar\eta_\ell,\, \bar\sigma_\ell \}$
or $\max\{ \bar\mu_\ell \bar\eta_\ell,\, \bar\tau_\ell \bar\sigma_\ell \}$.

Our next objective is to determine a suitable marking set that corresponds to the identified refinement type.
Depending on whether spatial or parametric refinement is selected,
this marking set consists of either interior edge midpoints of the current spatial mesh,
or multi-indices from the current reduced margin.
First, we perform D{\"o}rfler marking on each of the sets of error indicators associated with primal and dual approximations.
Having obtained two marking sets of the required refinement type, we then follow
the strategy from~\cite{ms2009} to select a single marking set out of the two---the one that has minimum cardinality.
This final marking set will feed into the refinement process in steps~(v) or~(vi) of Algorithm~\ref{algorithm}.
Alternative approaches to combining two marking sets into one final set were proposed and analyzed in~\cite{bet2011,fpz2016}
in the deterministic setting.
The associated marking strategies lead to the same convergence rate as the `minimum cardinality' strategy of~\cite{ms2009},
but require less adaptive iterations to reach a prescribed accuracy;
see the experiments in~\cite{fpz2016} for some deterministic test problems.

The above ideas 
are summarized in the following algorithm.

\begin{algorithm} \label{alg:marking}
\textbf{Input:}
error indicators
$\{ \mu_{\ell \ybar}(\xi) : \ybar \in \Colpts_\ell,\, \xi \in \NN_{\ell}^+ \}$,
$\{ \mu_{\ell \ybar} 
: \ybar \in \Colpts_\ell \}$,
$\{ \widetilde\tau_{\ell \nnu} : \nnu \in  \rmarg({\indset_\ell)}  \}$
and
$\{ \eta_{\ell \ybar}(\xi) : \ybar \in \Colpts_\ell,\, \xi \in \NN_{\ell}^+ \}$,
$\{ \eta_{\ell \ybar} 
: \ybar \in \Colpts_\ell \}$,
$\{ \widetilde\sigma_{\ell \nnu} : \nnu \in  \rmarg({\indset_\ell)}  \}$;
marking parameters $0 < \theta_\X, \theta_\Colpts \le 1$.
\begin{itemize}
\item[$\bullet$]
Calculate the cumulative error indicators
$\bar\mu_\ell,\, \bar\tau_\ell,\, \bar\eta_\ell,\, \bar\sigma_\ell$
given by~\eqref{eq:err:indicators} and~\eqref{eq:err:indicators:dual}.
\item[$\bullet$]
If \
$\bar\mu_\ell^2 + \bar\eta_\ell^2 \ge \bar\tau_\ell^2 + \bar\sigma_\ell^2,
$
then proceed as follows:
\begin{itemize}
\item[$\circ$]
set $\markindset_\ell := \emptyset$;

\item[$\circ$]
for each $\ybar \in \Colpts_\ell$,
determine $\MM^{(\mu)}_{\ell \ybar} \subseteq \NN_{\ell}^+$ and $\MM^{(\eta)}_{\ell \ybar} \subseteq \NN_{\ell}^+$ 
such that
\begin{equation} \label{eq:doerfler:separate1:primal}
 \theta_\X \, \sum_{\ybar \in \Colpts_\ell} \sum_{\xi \in \NN_{\ell}^+} \mu_{\ell \ybar}(\xi) \norm{L_{\ell \ybar}}{L^{2}_{\pi}(\G)} \le
 \sum_{\ybar \in \Colpts_\ell} \sum_{\xi \in \MM^{(\mu)}_{\ell \ybar}} \mu_{\ell \ybar}(\xi) \norm{L_{\ell \ybar}}{L^{2}_{\pi}(\G)},
\end{equation}
\begin{align} \label{eq:doerfler:separate1:dual}
 \theta_\X \, \sum_{\ybar \in \Colpts_\ell} \sum_{\xi \in \NN_{\ell}^+} \eta_{\ell \ybar}(\xi) \norm{L_{\ell \ybar}}{L^{2}_{\pi}(\G)}  \le
 \sum_{\ybar \in \Colpts_\ell} \sum_{\xi \in \MM^{(\eta)}_{\ell \ybar}} \eta_{\ell \ybar}(\xi) \norm{L_{\ell \ybar}}{L^{2}_{\pi}(\G)}
\end{align}
with cumulative cardinalities $\sum_{\ybar \in \Colpts_\ell} \#\MM^{(\mu)}_{\ell \ybar}$ and
$\sum_{\ybar \in \Colpts_\ell} \#\MM^{(\eta)}_{\ell \ybar}$ that are  minimized
over all the sets that satisfy~\eqref{eq:doerfler:separate1:primal} and \eqref{eq:doerfler:separate1:dual}, respectively;

\item[$\circ$]
set $\MM^{(\mu)}_{\ell} := \bigcup_{\ybar \in \Colpts_\ell} \MM^{(\mu)}_{\ell \ybar}$ and
$\MM^{(\eta)}_{\ell} := \bigcup_{\ybar \in \Colpts_\ell} \MM^{(\eta)}_{\ell \ybar}$;

\item[$\circ$]
if $\#\MM^{(\mu)}_{\ell} \,{\le}\,  \#\MM^{(\eta)}_{\ell}$, then choose $\MM_{\ell} := \MM^{(\mu)}_{\ell}$;
otherwise, choose $\MM_{\ell} := \MM^{(\eta)}_{\ell}$.
\end{itemize}
\item[$\bullet$]
Otherwise, i.e., if \ 
$\bar\mu_\ell^2 + \bar\eta_\ell^2 < \bar\tau_\ell^2 + \bar\sigma_\ell^2$,
proceed as follows:
\begin{itemize}
\item[$\circ$]
set $\MM_{\ell} := \emptyset$;
\item[$\circ$]
determine $\markindset^{(\tau)}_\ell \subseteq \rmarg(\indset_\ell) $ and $\markindset^{(\sigma)}_\ell \subseteq \rmarg(\indset_\ell) $
of minimal cardinality such that
\begin{equation*} 
   \theta_\Colpts \, \sum_{\nnu \in \rmarg(\indset_\ell) } \widetilde\tau_{\ell \nnu} \le
   \sum_{\nnu \in \markindset^{(\tau)}_\ell} \widetilde\tau_{\ell \nnu}
  \quad\text{and}\quad
  \theta_\Colpts \, \sum_{\nnu \in \rmarg(\indset_\ell) } \widetilde\sigma_{\ell \nnu} \le
  \sum_{\nnu \in \markindset^{(\sigma)}_\ell} \widetilde\sigma_{\ell \nnu};
\end{equation*}
\item[$\circ$]
if $\#\markindset^{(\tau)}_\ell \le \#\markindset^{(\sigma)}_\ell$, then choose $\markindset_\ell := \markindset^{(\tau)}_\ell$;
otherwise, choose $\markindset_\ell := \markindset^{(\sigma)}_\ell$.
\end{itemize}
\end{itemize}
\textbf{Output:}
$\MM_{\ell} \subseteq \NN_{\ell}^+$ and
$\markindset_\ell \subseteq  \rmarg(\indset_\ell)$.
\end{algorithm}

\subsection{Computation of $\Qbar(u_{\bullet}^{\rm SC},z_{\bullet}^{\rm SC})$} \label{sec:approx:goal}

An important detail in Algorithm~\ref{algorithm} that we need to address is
the computation of $\Qbar(u_{\bullet}^{\rm SC},z_{\bullet}^{\rm SC})$---the approximation of the functional value $\Q(u)$.
Recalling the definition of $\Qbar(u_{\bullet}^{\rm SC},z_{\bullet}^{\rm SC})$ in~\eqref{eq:goal:approx},
let us first address the computation of $\Q(\scsol)$ and $\F(\scsoldual)$.
In fact, it is easy to see that $\Q(\scsol) = \F(\scsoldual)$.
Indeed,
\begin{align*}
   \Q(\scsol) \reff{eq:goal}{=} \int_\Gamma Q(\scsol(\cdot,\y)) \dpi(\y) &\reff{eq:scfem:sol}{=}
   \int_\Gamma Q \bigg( \sum_{\ybar \in \Colpts_\bullet} \udotyh \LagrBasis{\bullet \ybar}{}(\y) \bigg) \dpi(\y) \\[4pt]
   &\refp{eq:scfem:sol}{=} 
   \sum_{\ybar \in \Colpts_\bullet} Q(\udotyh) \int_\Gamma \LagrBasis{\bullet \ybar}{}(\y) \dpi(\y) 
\end{align*}
and a similar calculation 
shows that
\begin{align*}
   \F(\scsoldual) = 
   \int_\Gamma \sum_{\ybar \in \Colpts_\bullet} F(\zdotyh) \LagrBasis{\bullet \ybar}{}(\y) \dpi(\y) = 
   \sum_{\ybar \in \Colpts_\bullet} F(\zdotyh) \int_\Gamma \LagrBasis{\bullet \ybar}{}(\y) \dpi(\y). 
\end{align*}
Now, observing that
\[
   Q(u_{\bullet \ybar}) \reff{eq:sample:fem:dual}{=} \By(u_{\bullet \ybar}, z_{\bullet \ybar})
   \reff{eq:sample:fem}{=}
   F(z_{\bullet \ybar})\quad \forall\, \ybar \in \Colpts_\bullet,
\]
we conclude that
\begin{align} \label{eq:Q(u^SC)}
   \Q(\scsol) = \sum_{\ybar \in \Ydot} \By(\udotyh, \zdotyh) \int_\Gamma \LagrBasis{\bullet \ybar}{}(\y) \dpi(\y) = \F(\scsoldual)
\end{align}
and therefore, the approximation $\Qbar(u_{\bullet}^{\rm SC},z_{\bullet}^{\rm SC})$ of the goal functional $\Q(u)$ can be calculated using
a simplified formula
\begin{align} \label{eq:goal:approx:simple}
   \Qbar(\scsol,\scsoldual) \reff{eq:goal:approx}{=}
   \Q(\scsol) + \F(\scsoldual) - \B(\scsol,\scsoldual) =
   2\Q(\scsol) - \B(\scsol,\scsoldual).
\end{align}
Here, $\Q(\scsol)$ is easily evaluated using the representation in~\eqref{eq:Q(u^SC)}:
while the spatial contributions (i.e., the bilinear forms) are evaluated in a standard way
(using the finite element coefficient vectors and the stiffness matrices associated with samples of the diffusion coefficient),
the integrals of $\LagrBasis{\bullet \ybar}{}$ are calculated by representing these polynomial functions
in terms of products of 1D Lagrange basis functions using the combination technique
\revision{(see, e.g.,~\cite{Wasilkowski95} and~\cite[section~3]{GuignardN18})}
and by using precomputed values for the integrals of 1D Lagrange basis functions
for a given type of collocation points (e.g., Leja or Clenshaw--Curtis points).

The computation of the term $\B(\scsol,\scsoldual)$ in~\eqref{eq:goal:approx:simple} is more involved.
Using the representations of $\scsol$, $\scsoldual$ in~\eqref{eq:scfem:sol},~\eqref{eq:scfem:sol:dual}, we have
\begin{align} \label{eq:goal:B-term}
   \B(\scsol,\scsoldual) &=
   \int_\Gamma \int_D a(x,\y) \nabla \scsol(x,\y) \cdot \nabla \scsoldual(x,\y) \dx \dpi(\y)
   \nonumber
   \\[4pt]
   &= \sum_{\ybar, \ybar' \in \Colpts_\bullet} \int_\Gamma \int_D a(x,\y)
        \nabla \udotyh(x) \cdot \nabla \zdotyhh(x) \LagrBasis{\bullet \ybar}{}(\y) \LagrBasis{\bullet \ybar'}{}(\y) \dx \dpi(\y).
\end{align}
Now, similarly to the calculation of $\Q(\scsol)$ in~\eqref{eq:Q(u^SC)}, our aim is to represent each term in the above sum
as a product of easily computable spatial and parametric integrals.
Doing this exactly is only possible in special cases.
For example, if the diffusion coefficient $a(x,\y)$ has an affine expansion in terms of parameters, i.e.,
\[
   a(x,\y) = a_0(x) + \sum_{m=1}^M y_m a_m(x),
\]
we obtain (here, we set $y_0 := 1$)
\begin{align*}
   \B(\scsol,\scsoldual) \,{=}\,
   \sum_{\y, \y' \in \Colpts_\bullet}
   \sum_{m=0}^M
   \bigg( \int_D a_m(x) \nabla\udotyh(x) \cdot \nabla\zdotyhh(x) \dx \bigg) 
   \bigg( \int_\Gamma y_m \LagrBasis{\bullet \ybar}{}(\y) \LagrBasis{\bullet \ybar'}{}(\y)  \dpi(\y) \bigg).
\end{align*}
The resulting spatial integrals are then evaluated in a standard way (as explained above for the calculation of $\Q(\scsol)$), and
the parametric integrals are again calculated from precomputed values of the corresponding integrals of polynomials in 1D.

In principle, such exact representations of $\B(\scsol,\scsoldual)$ in terms of computable spatial and parametric integrals
are possible for any other \emph{polynomial type} expansion of the diffusion coefficient, e.g., for the quadratic expansion
$a(x,\y) = \big(a_0(x) + \sum_{m=1}^M y_m a_m(x) \big)^2$.
In practice, however, calculating these integrals becomes computationally expensive very quickly.
Hence, even for nonaffine polynomial type expansions, it may be prudent to use the quadrature-based approach
that is presented next. 

For a general (nonaffine) diffusion coefficient $a(x,\y)$, we use a sparse grid-based quadrature rule to approximate
the integrals over $\Gamma$ in~\eqref{eq:goal:B-term}.
First, we note that using the quadrature based on the current set of collocation points $\Colpts_\bullet$
makes our use of the `correction term' $\F(\scsoldual) - \B(\scsol,\scsoldual)$ 
redundant.
Indeed, the approximation of integrals over $\Gamma$ in~\eqref{eq:goal:B-term} using the sparse grid $\Colpts_\bullet$
will result in the approximation of $\B(\scsol,\scsoldual)$ by $\F(\scsoldual)$.
Therefore, one has to use an enriched set of collocation points, 
e.g., we can use the set $\widehat\Colpts_\bullet = \Colpts_{\indset_\bullet \cup \rmarg(\indset_\bullet)}$.
Denoting by $\omega_{\ybarrr}$, $\ybarrr \in \widehat\Colpts_\bullet$, the associated quadrature weights,
we approximate the bilinear form $\B(\scsol,\scsoldual)$ given by~\eqref{eq:goal:B-term} as follows:
\begin{align*}
   \B(\scsol,\scsoldual) &\approx
   \sum_{\ybarrr \in \widehat\Colpts_\bullet}
   \omega_{\ybarrr}
   \sum_{\ybar, \ybar' \in \Colpts_\bullet}
   \bigg[ \int_D a(x,\ybarrr) \nabla \udotyh(x) \cdot \nabla \zdotyhh(x) \dx \bigg]\,
   \LagrBasis{\bullet \ybar}{}(\ybarrr) \LagrBasis{\bullet \ybar'}{}(\ybarrr)\\[4pt]
   &=
   \bigg( \sum_{\ybarrr \in \Colpts_\bullet} + \sum_{\ybarrr \in \widehat\Colpts_\bullet \setminus \Colpts_\bullet} \bigg)\,
   \omega_{\ybarrr}
   \sum_{\ybar, \ybar' \in \Colpts_\bullet}
   B_{\ybarrr}(\udotyh, \zdotyhh) \,
   \LagrBasis{\bullet \ybar}{}(\ybarrr) \LagrBasis{\bullet \ybar'}{}(\ybarrr),
\end{align*}
where we used the notation of the samplewise bilinear form (cf.~\eqref{eq:sampled:bilinear:form}).
Hence, recalling that $\LagrBasis{\bullet \ybar}{}(\ybarrr) = \delta_{\ybar \ybarrr}$ and
$\LagrBasis{\bullet \ybar'}{}(\ybarrr) = \delta_{\ybar' \ybarrr}$ for any $\ybar,\, \ybar',\, \ybarrr \in \Colpts_\bullet$,
we obtain the following formula for (approximate) calculation of the bilinear form $\B(\scsol,\scsoldual)$:
\begin{align*}
   \B(\scsol,\scsoldual) &\approx
   \sum_{\ybarrr \in \widehat\Colpts_\bullet}
   \omega_{\ybarrr} \, \widetilde{B}_{\ybarrr} (\scsol, \scsoldual),
\end{align*}
where
\begin{align*}
   \widetilde{B}_{\ybarrr} (\scsol, \scsoldual) :=
   \begin{cases}
      B_{\ybarrr}(u_{\bullet\ybarrr}, z_{\bullet\ybarrr})
      & \quad\text{if $\ybarrr \in \Colpts_\bullet$},
      \\[4pt]
      \Sum_{\ybar, \ybar' \in \Colpts_\bullet}
      B_{\ybarrr}(\udotyh, \zdotyhh)\,
      \LagrBasis{\bullet \ybar}{}(\ybarrr) \LagrBasis{\bullet \ybar'}{}(\ybarrr)
      & \quad\text{if $\ybarrr \in \widehat\Colpts_\bullet \setminus \Colpts_\bullet$}.
   \end{cases}
\end{align*}

\section{Goal-oriented adaptive SC-FEM with nonlinear goal functionals} \label{sec:goascfem:nl}

In this section, we aim to extend the error estimation strategy and adaptive algorithm 
presented in section~\ref{sec:goascfem} to a larger class of (possibly nonlinear) goal functionals.
Specifically, let the quantity of interest $Q(u(\cdot,\y))$ be obtained using a functional $Q: \X \to \R$ and assume that $Q \in C^1$,
i.e., $Q$ is G{\^ a}teaux differentiable and its G{\^ a}teaux derivative $Q': \X \to \X^*$ is continuous.
As before, the goal functional $\Q: \V \to \R$ is defined by~\eqref{eq:goal}.
Thus, $\Q$ is also G{\^ a}teaux differentiable and its G{\^ a}teaux derivative $\Q': \V \to \V^*$ is continuous.
We will denote by $\langle \cdot, \cdot \rangle_{D}$ the duality pairing between $\X$ and its dual $\X^*$;
similarly, $\langle \cdot, \cdot \rangle_{D \times \G}$ denotes the duality pairing between $\V$ and $\V^*$.

\subsection{Dual problems and the dual SC-FEM approximation} \label{sec:dual:problems:nl}

Following the same approach as in~\S\ref{sec:goal:error:estimation} and using the ideas from~\cite{bpr2023},
let us first formulate dual problems that facilitate the goal-oriented error estimation.

Recall that for $\pi$-almost all $\bar\y \in \Gamma$, $u_{\bar\y} \in \X$ denotes the samplewise solution satisfying~\eqref{eq:weak:sampled},
and $u_{\bullet \bar\y} \in \X_\bullet$ denotes its Galerkin finite element approximation satisfying~\eqref{eq:sample:fem}.
According to the fundamental theorem of calculus there holds:
\begin{equation*}
    Q (u_{\bar\y}) - Q(u_{\bullet \bar\y}) =
    \int_0^1 \langle Q'(u_{\bullet \bar\y} + t(u_{\bar\y} - u_{\bullet \bar\y})), u_{\bar\y} - u_{\bullet \bar\y} \rangle_{D} \,\dt =:
    \langle Q^*_u(u_{\bullet \bar\y}), u_{\bar\y} - u_{\bullet \bar\y} \rangle_{D}.
\end{equation*}
Since formally speaking $\|Q^*_u(u_{\bullet\bar\y}) - Q'(u_{\bullet\bar\y}) \|_{\X^*} \xrightarrow{} 0$ as
$u_{\bullet\bar\y} \to u_{\bar\y}$, we consider the following \emph{samplewise} dual formulation for $\pi$-almost all $\bar\y \in \Gamma$
(cf.~\eqref{eq:dual:sampled}):
find $z_{\ybar}(x) := z(x,\ybar) \in \X$ satisfying
\begin{align} \label{eq:sample:dual:nl}
   B_{\bar\y}(v,  z_{\bar\y}) = \langle Q'(u_{\bullet\bar\y}), v \rangle_{D} \quad \forall v\in \X.
\end{align}
This formulation and the definition of the goal functional $\Q: \V \to \R$ in~\eqref{eq:goal} yield the
following \emph{combined} dual formulation 
(cf.~\eqref{eq:weak:dual}):
find $z \in \V$ such that
\begin{equation} \label{eq:weak:dual:nl}
    \B(v, z) = \langle \Q'(u), v\rangle_{D\times\G} \quad \text{for all} \; v \in \V.
\end{equation}
The dual formulations in~\eqref{eq:sample:dual:nl} and \eqref{eq:weak:dual:nl} are well-posed due to the Lax--Milgram lemma.

In the same way as we did for the \emph{samplewise} dual formulation in \S~\ref{sec:goal:error:estimation},
we discretize~\eqref{eq:sample:dual:nl} at each collocation point $\ybar \in \Colpts_\bullet$ by the Galerkin FEM and
build the Lagrange interpolant out of the resulting Galerkin approximations.
Specifically, for each $\ybar \in \Colpts_\bullet$, we 
denote by $z_{\bullet \ybar} \in \X_{\bullet}$ the Galerkin finite element approximation satisfying
\begin{align} \label{eq:sample:fem:dual:nl}
   B_{\bar\y}(v,  z_{\bullet \bar\y}) = \langle Q'(u_{\bullet \bar\y}), v \rangle_{D} \quad \forall v\in \X_\bullet.
\end{align}
Given the set $\big\{z_{\bullet \ybar}(x) : \ybar \in \Colpts_\bullet\big\}$,
we construct the dual SC-FEM approximation $\scsoldual$ as in~\eqref{eq:scfem:sol:dual}.

Now we can define the approximation of $\Q(u)$ in the general case considered here.
In fact, we will use the same definition as in the case of linear goal functionals (cf.~\eqref{eq:goal:approx}):
\[
   \Q(u) \approx \Qbar(\scsol, \scsoldual) := \Q(\scsol)+\B(u-\scsol, \scsoldual).
\]

\subsection{Goal-oriented error estimation}  \label{sec:error:estimation:nl}

In order to estimate the approximation error $\big| \Q(u) - \Qbar(\scsol, \scsoldual) \big|$,
we assume that there exists a constant $C_{\text{goal}} \ge 0$ such that
\begin{equation} \label{ineq:gateaux}
    | \langle \Q'(v) - \Q'(w), z \rangle_{D\times \G}| \le
    C_{\text{goal}} \, \enorm[\big]{v - w}{}{} \enorm[\big]{z}{}{} \quad \text{for all } v, w, z \in \V.
\end{equation}
%
%
Furthermore, analogously to~\eqref{eq:saturation}, we assume that 
there exists a constant $q_{\rm sat} \in (0,1)$ independent of discretization~parameters such that
\begin{equation} \label{eq:saturation:dual:nl}
   \norm{z - \scsoldualhat}{} \le q_{\rm sat} \norm{z - \scsoldual}{},
\end{equation}
where
$z \in \V$ solves~\eqref{eq:weak:dual:nl},
the dual SC-FEM approximation $\scsoldual$ is defined as described in~\S\ref{sec:dual:problems:nl} above, and
$\scsoldualhat$ is the enhanced approximation of $\scsoldual$ constructed as described in~\S\ref{sec:error:estimation}.
Then, similarly to~\eqref{eq:error:estimate} and~\eqref{eq:error:estimate:dual}, we obtain
\begin{equation} \label{eq:error:estimate:dual:nl}
   \norm{z - \scsoldual}{} \le \big(1 - q_{\rm sat}\big)^{-1} \, 
   \big(\eta_\bullet + \sigma_\bullet\big),
\end{equation}
where $\eta_\bullet$ and $\sigma_\bullet$ are defined by~\eqref{eq:dual:estimates}
(e.g., $\eta_\bullet = \mu_\bullet[z]$, that is $\eta_\bullet$ is given by the right-hand side of~\eqref{eq:spatial:estimate}
with $u$ replaced by $z$;
in particular, $\eta_\bullet$ employs the enhanced samplewise Galerkin approximations $\widehat z_{\bullet\ybar}$
that satisfy~\eqref{eq:sample:fem:dual:nl} with $\X_\bullet$ replaced by $\widehat\X_\bullet$).

We are now ready to prove the goal-oriented a posteriori error estimate.

\begin{theorem} \label{thm:error:estimate}
Suppose the saturation assumptions~\eqref{eq:saturation} and~\eqref{eq:saturation:dual:nl} hold for the
primal solution $u$ and for the dual solution 
$z$, respectively.
In addition, assume that the goal functional $\Q$ satisfies~\eqref{ineq:gateaux}.
Then, the following error estimate holds:
\begin{equation} \label{eq:goal:error:estimate:nl}
    |\Q (u) - \Qbar(\scsol, \scsoldual)|
    \lesssim
    (\mu_\bullet+\tau_\bullet)(\mu_\bullet+\tau_\bullet+\eta_\bullet +\sigma_\bullet),
\end{equation}
where $\mu_\bullet, \;\tau_\bullet,\; \eta_\bullet,\; \sigma_\bullet$
are defined in~\eqref{eq:spatial:estimate},~\eqref{eq:param:estimate},~\eqref{eq:dual:estimates} and
the hidden constant depends only on
$a_{\max} > 0$ in~\eqref{eq:amin:amax}, $C_{\rm goal} \ge 0$ in~\eqref{ineq:gateaux},
and $q_{\rm sat} \in [0,1)$ in~\eqref{eq:saturation},~\eqref{eq:saturation:dual:nl}.
\end{theorem}

\begin{proof}
The fundamental theorem of calculus implies that
\begin{align*} 
    \Q(u) & \,{-}\, \Q(\scsol)
    \refp{eq:weak:dual:nl}{=}
    \int_0^1 \langle\Q'(\scsol + t(u - \scsol)), u-\scsol\rangle_{D\times\G}\, \dt
    \nonumber
    \\[4pt]
    &\refp{eq:weak:dual:nl}{=}
    \int_0^1 \langle\Q'(\scsol + t(u - \scsol)) - \Q'(u), u-\scsol\rangle_{D\times\G}\, \dt
    \,+\, \langle \Q'(u), u-\scsol \rangle_{D\times\G}
    \nonumber
    \\[4pt]
    &\reff{eq:weak:dual:nl}{=}
    \int_0^1 \langle\Q'(\scsol + t(u - \scsol)) -\Q'(u), u-\scsol\rangle_{D\times\G}\, \dt
    \,+\, \B(u - \scsol, z).
\end{align*}
Therefore,
\begin{align*}
    & \Q (u) - \Qbar(\scsol, \scsoldual) \reff{eq:goal:approx}{=} \Q(u) - \Q(\scsol) - \B(u-\scsol, \scsoldual)
    \nonumber
    \\[4pt] 
    &  = 
    \int_0^1 \langle\Q'(\scsol + t(u - \scsol)) - \Q'(u), u-\scsol\rangle_{D\times\G}\, \dt
    \,+\, \B(u - \scsol, z - \scsoldual).
\end{align*}
Using inequality~\eqref{ineq:gateaux}, we obtain
\begin{align*}
    \int_0^1 \langle\Q'(\scsol + t(u - & \scsol)) - \Q'(u), u-\scsol\rangle_{D\times\G}\, \dt
    \nonumber 
    \\[4pt]   
    & \le C_{\text{goal}} \int_0^1 (1-t) \, \enorm[\big]{u - \scsol}{}{}^2 \;\dt
     = \frac{1}{2} C_{\text{goal}} \enorm[\big]{u - \scsol}{}{}^2
\end{align*}
and, hence, we derive the following estimate of the error in the goal functional:
\begin{equation} \label{ineq:goal:error:nl}
    |\Q (u) - \Qbar(\scsol, \scsoldual)| \le \frac{1}{2} C_{\text{goal}} \, \enorm[\big]{u - \scsol}{}{}^2 
    + \enorm[\big]{u - \scsol}{}{} \enorm[\big]{z - \scsoldual}{}{}.
\end{equation}
Finally, using the norm equivalence~\eqref{eq:norm:equiv} and
the a posteriori error estimates~\eqref{eq:error:estimate} and~\eqref{eq:error:estimate:dual:nl},
we conclude from~\eqref{ineq:goal:error:nl} that
\begin{equation*} 
    |\Q (u) - \Qbar(\scsol, \scsoldual)| \lesssim (\mu_\bullet+\tau_\bullet)^2 + (\mu_\bullet+\tau_\bullet)(\eta_\bullet +\sigma_\bullet).
\end{equation*}
This implies the required error estimate~\eqref{eq:goal:error:estimate:nl}.
\end{proof}

\begin{remark}
Assumption~\eqref{ineq:gateaux}
holds true at least for linear and quadratic goal functionals.
Indeed, if $\Q \in \V^*$, then $\langle \Q'(w), v \rangle_{D\times \G} = g(v)$ for all $v,\,w \in \V$ and for some $g \in \V^*$.
Consequently, in this case,
the dual problem in~\eqref{eq:weak:dual:nl} reduces to the dual formulation in~\eqref{eq:weak:dual}. 
Moreover, inequality~\eqref{ineq:gateaux} is satisfied with $C_{\text{goal}} = 0$, and
the error estimates~\eqref{eq:goal:error:estimate:nl} and~\eqref{ineq:goal:error:nl}
reduce to the estimates~\eqref{eq:goal:error:estimate} obtained in the case of linear goal functionals.
Furthermore, for the quadratic goal functional $\Q(u) = b(u,u)$, where $b: \V \times \V \to  \R$ is a
continuous bilinear form, we have
\begin{equation*}
    \langle \Q'(v) - \Q'(w), z \rangle_{D\times \G} = b(v-w, z) + b(z, v-w) \quad \forall\, v,\,w,\,z \in \V.
\end{equation*}
Therefore, inequality~\eqref{ineq:gateaux} is valid in this case with $C_{\text{goal}}$ depending
only on the continuity constant for $b(\cdot,\cdot)$.
\end{remark}

\subsection{Modifications to the adaptive algorithm} \label{sec:modifications:nl}

In the light of the developments in~\S\ref{sec:dual:problems:nl} and~\S\ref{sec:error:estimation:nl},
we need to replace the stopping criterion in step~(vii) of Algorithm~\ref{algorithm}
(according to Theorem~\ref{thm:error:estimate}, the new stopping criterion is given by
$(\mu_\ell + \tau_\ell) \, (\mu_\ell + \tau_\ell + \eta_\ell + \sigma_\ell)< {\tt tol}$)
and we need to make two further modifications to Algorithm~\ref{algorithm}.
Firstly, the samplewise dual Galerkin approximations
$\big\{ z_{\ell \ybar} \in \X_{\ell} : \ybar \in \Colpts_{\indset_\ell \cup \rmarg(\indset_\ell)} \big\}$
are now computed by solving~\eqref{eq:sample:fem:dual:nl}
instead of~\eqref{eq:sample:fem:dual}; 
see step~(i) of Algorithm~\ref{algorithm}.
As a consequence, the corresponding local spatial error indicators are defined as follows (cf.~\eqref{eq:2level:local:indicator:dual}):
\begin{equation} \label{eq:2level:local:indicator:dual:nl}
   \eta_{\ell \ybar}(\xi) :=
   \frac{\big| \langle Q'(u_{\ell \ybar}), \widehat\varphi_{\ell,\xi} \rangle_{D } -
            B_{\ybar}(\widehat\varphi_{\ell,\xi}, z_{\ell \ybar}) \big|}
           {\norm{\widehat\varphi_{\ell,\xi}}{\X}}\quad
   \forall\,\xi \in \NN_{\ell}^+,\ \ \forall\,\ybar \in \Colpts_\ell.
\end{equation}
We also emphasize the fact that in the case of nonlinear goal functionals,
the right-hand side of the samplewise discrete dual problem~\eqref{eq:sample:fem:dual:nl}
depends on a sample of the primal Galerkin solution.
Therefore, unlike in the linear case, the samplewise primal and dual discrete formulations~\eqref{eq:sample:fem}
and~\eqref{eq:sample:fem:dual:nl} are not independent of each other.
Consequently, the samplewise Galerkin approximations $u_{\ell \ybar}$ and $z_{\ell \ybar}$ for each $\ybar \in \Colpts_\ell$
need to be computed sequentially:  first the primal approximation $u_{\ell \ybar}$ satisfying~\eqref{eq:sample:fem}
and then the dual approximation $z_{\ell \ybar}$ satisfying~\eqref{eq:sample:fem:dual:nl}.

The second modification to Algorithm~\ref{algorithm} concerns the marking criterion.
In view of the error estimate in~\eqref{eq:goal:error:estimate:nl}, the adaptive algorithm should employ a modified marking criterion
that ensures a reduction of either the primal error estimate $(\mu_\ell+\tau_\ell)$
or the combined primal-dual error estimate $(\mu_\ell+\tau_\ell+\eta_\ell +\sigma_\ell)$.
However, for identifying the refinement type (spatial versus parametric),
we can still use the quantity in~\eqref{eq:identify:ref:type} as in the case of linear goal functionals.
This is now motivated by the following inequality (cf.~\eqref{eq:ref:type:motiv}):
\begin{align*} 
   \big(\mu_\ell + \tau_\ell\big) \big(\mu_\ell + \tau_\ell + \eta_\ell + \sigma_\ell\big)
   & \le
   \big(\bar\mu_\ell + \bar\tau_\ell\big)^2 + \big(\bar\mu_\ell + \bar\tau_\ell\big) \big(\bar\eta_\ell + \bar\sigma_\ell\big)
   \nonumber
   \\[3pt]
   &\le
   \big(3\bar\mu_\ell^2 + \bar\eta_\ell^2\big) + \big(3\bar\tau_\ell^2 + \bar\sigma_\ell^2\big)
   \le
   3 \big[ \big(\bar\mu_\ell^2 + \bar\eta_\ell^2\big) + \big(\bar\tau_\ell^2 + \bar\sigma_\ell^2\big) \big].
\end{align*}
The following algorithm specifies the required marking criterion to be used in step~(iv)
of Algorithm~\ref{algorithm} in the case of the nonlinear goal functional~$\Q$.

\begin{algorithm} \label{alg:marking:nl}
\textbf{Input:}
error indicators
$\{ \mu_{\ell \ybar}(\xi) : \ybar \in \Colpts_\ell,\, \xi \in \NN_{\ell}^+ \}$,
$\{ \mu_{\ell \ybar} 
: \ybar \in \Colpts_\ell \}$,
$\{ \widetilde\tau_{\ell \nnu} : \nnu \in  \rmarg({\indset_\ell)}  \}$
and
$\{ \eta_{\ell \ybar}(\xi) : \ybar \in \Colpts_\ell,\, \xi \in \NN_{\ell}^+ \}$,
$\{ \eta_{\ell \ybar} 
: \ybar \in \Colpts_\ell \}$,
$\{ \widetilde\sigma_{\ell \nnu} : \nnu \in  \rmarg({\indset_\ell)}  \}$;
marking parameters $0 < \theta_\X, \theta_\Colpts \le 1$.
\begin{itemize}
\item[$\bullet$]
Calculate the cumulative error indicators
$\bar\mu_\ell,\, \bar\tau_\ell,\, \bar\eta_\ell,\, \bar\sigma_\ell$
given by~\eqref{eq:err:indicators} and~\eqref{eq:err:indicators:dual}.
\item[$\bullet$]
If \
$\bar\mu_\ell^2 + \bar\eta_\ell^2 \ge \bar\tau_\ell^2 + \bar\sigma_\ell^2 $,
then proceed as follows:
\begin{itemize}
\item[$\circ$]
set $\markindset_\ell := \emptyset$;
\item[$\circ$]
for each $\ybar \in \Colpts_\ell$,
determine $\MM^{(1)}_{\ell \ybar} \subseteq \NN_{\ell}^+$ and $\MM^{(2)}_{\ell \ybar} \subseteq \NN_{\ell}^+$
such that
\begin{equation} \label{eq:doerfler:separate1:primal:nl}
 \theta_\X \, \sum_{\ybar \in \Colpts_\ell} \sum_{\xi \in \NN_{\ell}^+} \mu_{\ell \ybar}(\xi) \norm{L_{\ell \ybar}}{L^{2}_{\pi}(\G)} \le
 \sum_{\ybar \in \Colpts_\ell} \sum_{\xi \in \MM^{(1)}_{\ell \ybar}} \mu_{\ell \ybar}(\xi) \norm{L_{\ell \ybar}}{L^{2}_{\pi}(\G)},
\end{equation}
\begin{align} \label{eq:doerfler:separate1:dual:nl}
 & \theta_\X \, \sum_{\ybar \in \Colpts_\ell} \sum_{\xi \in \NN_{\ell}^+} (\mu_{\ell \ybar}(\xi)+\eta_{\ell \ybar}(\xi)) \norm{L_{\ell \ybar}}{L^{2}_{\pi}(\G)} \nonumber\\[4pt]
 & \qquad\qquad\le
 \sum_{\ybar \in \Colpts_\ell} \sum_{\xi \in \MM^{(2)}_{\ell \ybar}} (\mu_{\ell \ybar}(\xi)+\eta_{\ell \ybar}(\xi)) \norm{L_{\ell \ybar}}{L^{2}_{\pi}(\G)}
\end{align}
with cumulative cardinalities $\sum_{\ybar \in \Colpts_\ell} \#\MM^{(1)}_{\ell \ybar}$ and
$\sum_{\ybar \in \Colpts_\ell} \#\MM^{(2)}_{\ell \ybar}$ that are  minimized
over all the sets that satisfy~\eqref{eq:doerfler:separate1:primal:nl} and \eqref{eq:doerfler:separate1:dual:nl}, respectively;

\item[$\circ$]
set $\MM^{(1)}_{\ell} := \bigcup_{\ybar \in \Colpts_\ell} \MM^{(1)}_{\ell \ybar}$ and
$\MM^{(2)}_{\ell} := \bigcup_{\ybar \in \Colpts_\ell} \MM^{(2)}_{\ell \ybar}$;

\item[$\circ$]
if $\#\MM^{(1)}_{\ell} \le  \#\MM^{(2)}_{\ell}$, then choose $\MM_{\ell} := \MM^{(1)}_{\ell}$;
otherwise, choose $\MM_{\ell} := \MM^{(2)}_{\ell}$.
\end{itemize}
\item[$\bullet$]
Otherwise, i.e., if
$\bar\mu_\ell^2 + \bar\eta_\ell^2 < \bar\tau_\ell^2 + \bar\sigma_\ell^2 $,
proceed as follows:
\begin{itemize}
\item[$\circ$]
set $\MM_{\ell} := \emptyset$;
\item[$\circ$]
determine $\markindset^{(1)}_\ell \subseteq \rmarg(\indset_\ell) $ and $\markindset^{(2)}_\ell \subseteq \rmarg(\indset_\ell) $
of minimal cardinality such that
\begin{equation*} 
   \qquad
   \theta_\Colpts \, \sum_{\nnu \in \rmarg(\indset_\ell) } \widetilde\tau_{\ell \nnu} \le
   \sum_{\nnu \in \markindset^{(1)}_\ell} \widetilde\tau_{\ell \nnu}
   \quad\text{and}\quad
   \theta_\Colpts \, \sum_{\nnu \in \rmarg(\indset_\ell) } (\widetilde\tau_{\ell \nnu}+\widetilde\sigma_{\ell \nnu}) \le
   \sum_{\nnu \in \markindset^{(2)}_\ell} (\widetilde\tau_{\ell \nnu}+\widetilde\sigma_{\ell \nnu}).
\end{equation*}
\item[$\circ$]
if $\# \markindset^{(1)}_\ell \le \#\markindset^{(2)}_\ell$, then choose $\markindset_\ell := \markindset^{(1)}_\ell$;
otherwise, choose $\markindset_\ell := \markindset^{(2)}_\ell$.
\end{itemize}
\end{itemize}
\textbf{Output:}
$\MM_{\ell} \subseteq \NN_{\ell}^+$ and
$\markindset_\ell \subseteq  \rmarg(\indset_\ell)$.
\end{algorithm}

The final comment that is due in this section concerns the computation of $\Qbar(u_{\bullet}^{\rm SC},z_{\bullet}^{\rm SC})$
given by~\eqref{eq:goal:approx}.
Note that the components $\F(\scsoldual)$ and $\B(\scsol, \scsoldual)$ in the representation
of $\Qbar(u_{\bullet}^{\rm SC},z_{\bullet}^{\rm SC})$ are computed in exactly the same way as described in section~\ref{sec:approx:goal},
whereas the calculation of $\Q(\scsol)$ heavily depends on the specific form of the nonlinear functional $\Q$.
Therefore, we will address the computation of $\Q(\scsol)$ for some representative examples of nonlinear functionals when discussing
setups of test problems in section~\ref{sec:numerics}.
We emphasize that, unlike for linear goal functionals, $\F(\scsoldual) \ne \Q(\scsol)$ in the case of nonlinear $\Q$ and, thus,
all three components $\F(\scsoldual)$, $\Q(\scsol)$, and $\B(\scsol, \scsoldual)$ in the representation
of $\Qbar(u_{\bullet}^{\rm SC},z_{\bullet}^{\rm SC})$ need to be calculated in this case.

\section{Numerical experiments} \label{sec:numerics}

In this section, to demonstrate the performance of Algorithm~\ref{algorithm}, we report the results of numerical experiments
for four representative examples of linear and nonlinear goal functionals.
The computations were performed using the open source MATLAB toolbox Adaptive ML-SCFEM~\cite{BespalovSX_a_ml_scfem}.

The following settings and parameters will be the same across all four test setups:
\begin{itemize}

\revision{\item
the parameters $y_m, \; m=1, \dots, M$, are the images of uniformly distributed independent mean-zero random variables, so that
$\dpi_m(y_m) = \frac{1}{2} \mathrm{d}y_m$;}

\item
the employed sparse grids are based on Clenshaw--Curtis quadrature points;

\item
we represent the diffusion coefficient $a(x, \y)$ in terms of the affine-parametric function
$h(x, \y)$ \revision{introduced in~\cite{egsz14} and}
defined as
\begin{equation} \label{kll}
  h(x, \y) := h_0(x) + \sum_{m = 1}^M h_m(x) \, y_m,\quad
   x \in D,\ \y \in \Gamma,
\end{equation}
where the expansion coefficients $h_m$, $m =0, 1, \dots, M$,  are chosen
to represent planar Fourier modes of increasing total order;
specifically, we set $h_0(x) := 1$,
\begin{equation} \label{diff_coeff_Fourier}
  h_m(x) := 0.547\, m^{-2} \cos(2\pi\beta_1(m)\,x_1) \cos(2\pi\beta_2(m)\,x_2),\  x=(x_1,x_2) \in D
\end{equation}
for $m \in \N$, and order the modes so that 
\begin{equation}
  \beta_1(m) = m - k(m)(k(m)+1)/2\ \ \hbox{and}\ \ \beta_2(m) =k(m) - \beta_1(m)
\end{equation}
with $k(m) = \lfloor -1/2 + \sqrt{1/4+2m}\rfloor$;

\item
the stopping tolerance is set to ${\tt tol} = 10^{-5}$ (see step~(vii) of Algorithm~\ref{algorithm});

\item
the marking parameters $\theta_\X,\, \theta_\Colpts$ in Algorithms~\ref{alg:marking} and~\ref{alg:marking:nl} are both set to $0.3$.
\end{itemize}

In setups~1 and~2 described below, we consider linear goal functionals,
whereas in setups~3 and~4, the goal functionals are nonlinear.
Consequently, in setups~1 and~2, we employ Algorithm~\ref{algorithm} with the marking strategy as in Algorithm~\ref{alg:marking}
and evaluate $\Q(\scsol)$ and $\Qbar(u_{\bullet}^{\rm SC},z_{\bullet}^{\rm SC})$ as described in section~\ref{sec:approx:goal}.
On the other hand, in setups~3 and~4, we use a modification to Algorithm~\ref{algorithm} as explained in~\S\ref{sec:modifications:nl}
with the marking criterion in Algorithm~\ref{alg:marking:nl}, and we explain how we compute $\Q(\scsol)$ in each case
(recall that the other two components in the representation of $\Qbar(u_{\bullet}^{\rm SC},z_{\bullet}^{\rm SC})$
are computed in exactly the same way as in the case of linear functionals).

\smallskip

\emph{Setup 1: expectation of an integral over a spatial subdomain}.
In the first test case, we set $f(x) \equiv 1$ and consider
the model problem~\eqref{eq:pde:strong} on the L-shaped domain $D = (-1, 1)^2 \setminus (-1, 0]^2$
with affine-parametric diffusion coefficient, i.e., $a(x, \y) := h(x,\y)$, where $h(x,\y)$ is given by~\eqref{kll} \revision{with $M=4$}.
The goal functional is given by the expectation of the integral over the square subdomain $S := (0.25, 0.75)^2 \subset D$ as follows:
\begin{equation*}
    \Q(u) = \int_{\G} \int_D q(x) u(x, \y) \dx \, \dpi(\y),
\end{equation*}
where $q(x) = \chi_S/|S|$ and $\chi_{S}: D \to \{0, 1\}$ is the characteristic function of $S$.
The subdomain $S$ and the initial mesh on $D$ are shown in the Figure~\ref{fig:meshes:init}.

\begin{figure}[!thp]
\begin{subfigure}{.24\textwidth}
    \centering
    \caption*{Setup 1}
    \includegraphics[width = .95\linewidth,  trim=0.5cm 5.5cm 1cm 5.6cm,clip]{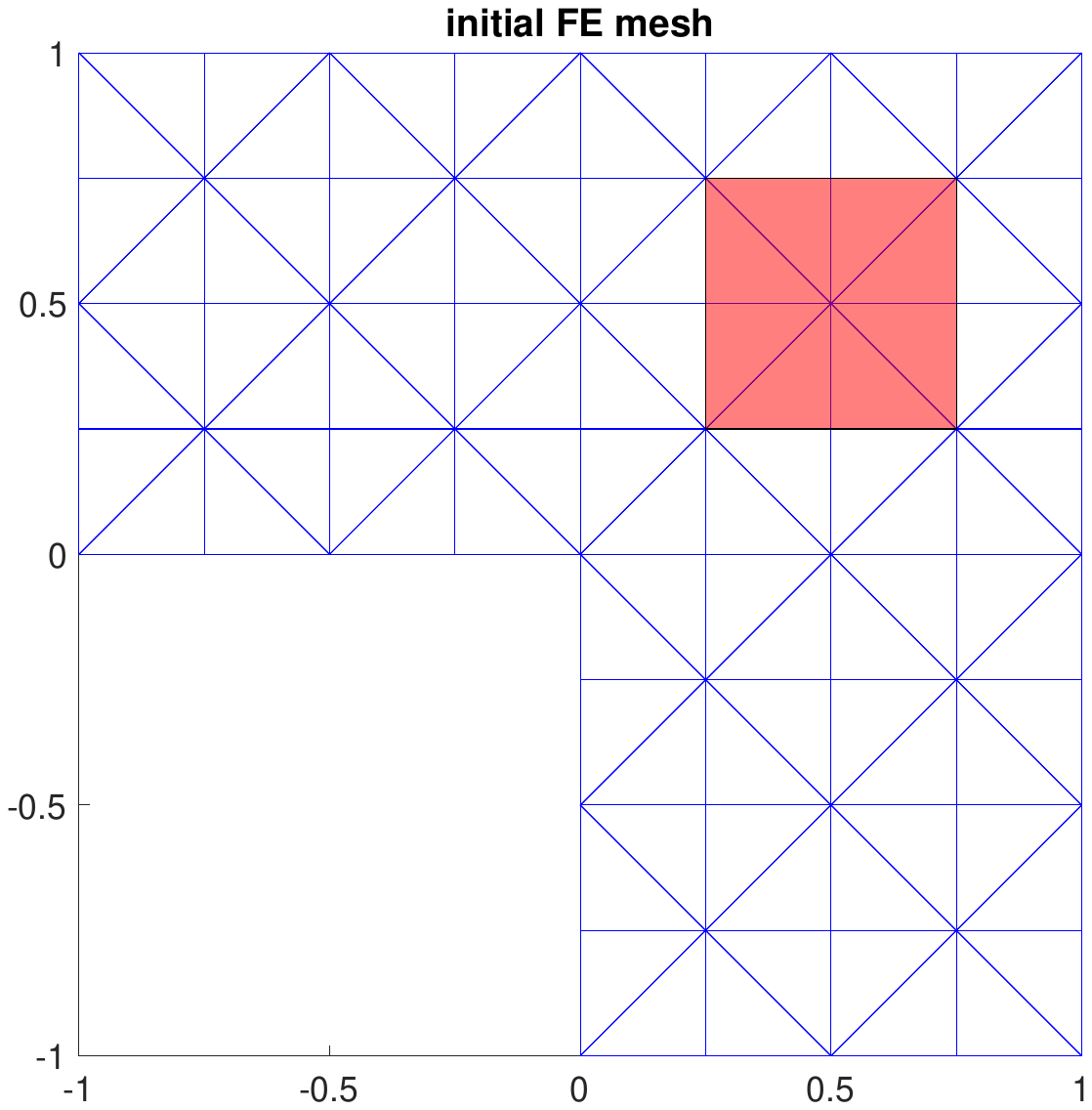}
\end{subfigure}
\begin{subfigure}{.24\textwidth}
    \centering
    \caption*{Setup 2}
    \includegraphics[width = .95\linewidth,  trim=0.5cm 5.5cm 1cm 5.6cm,clip]{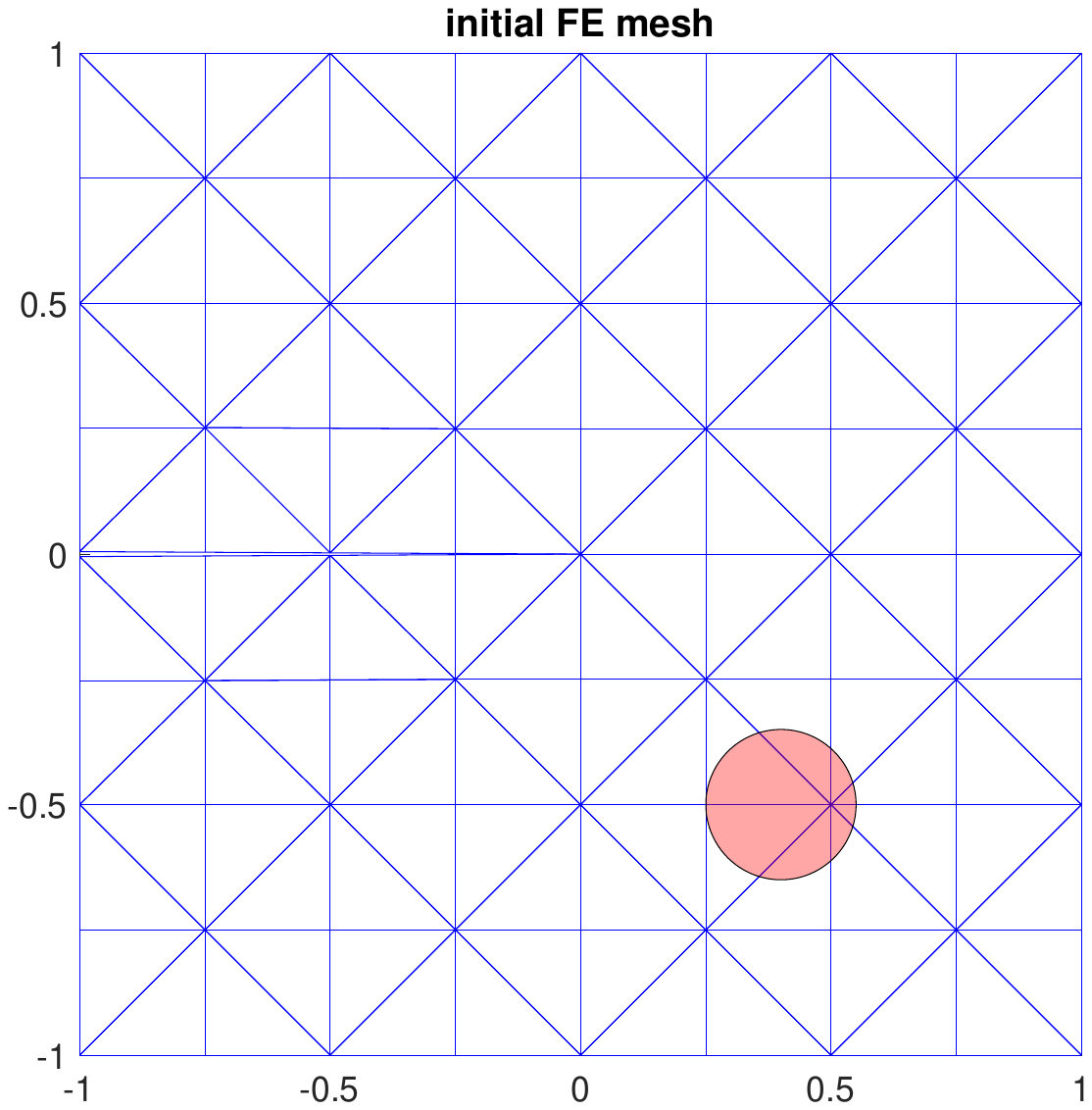}
\end{subfigure}
\begin{subfigure}{.24\textwidth}
    \centering
    \caption*{Setup 3}
    \includegraphics[width = .95\linewidth,  trim=0.5cm 5.5cm 1cm 5.6cm,clip]{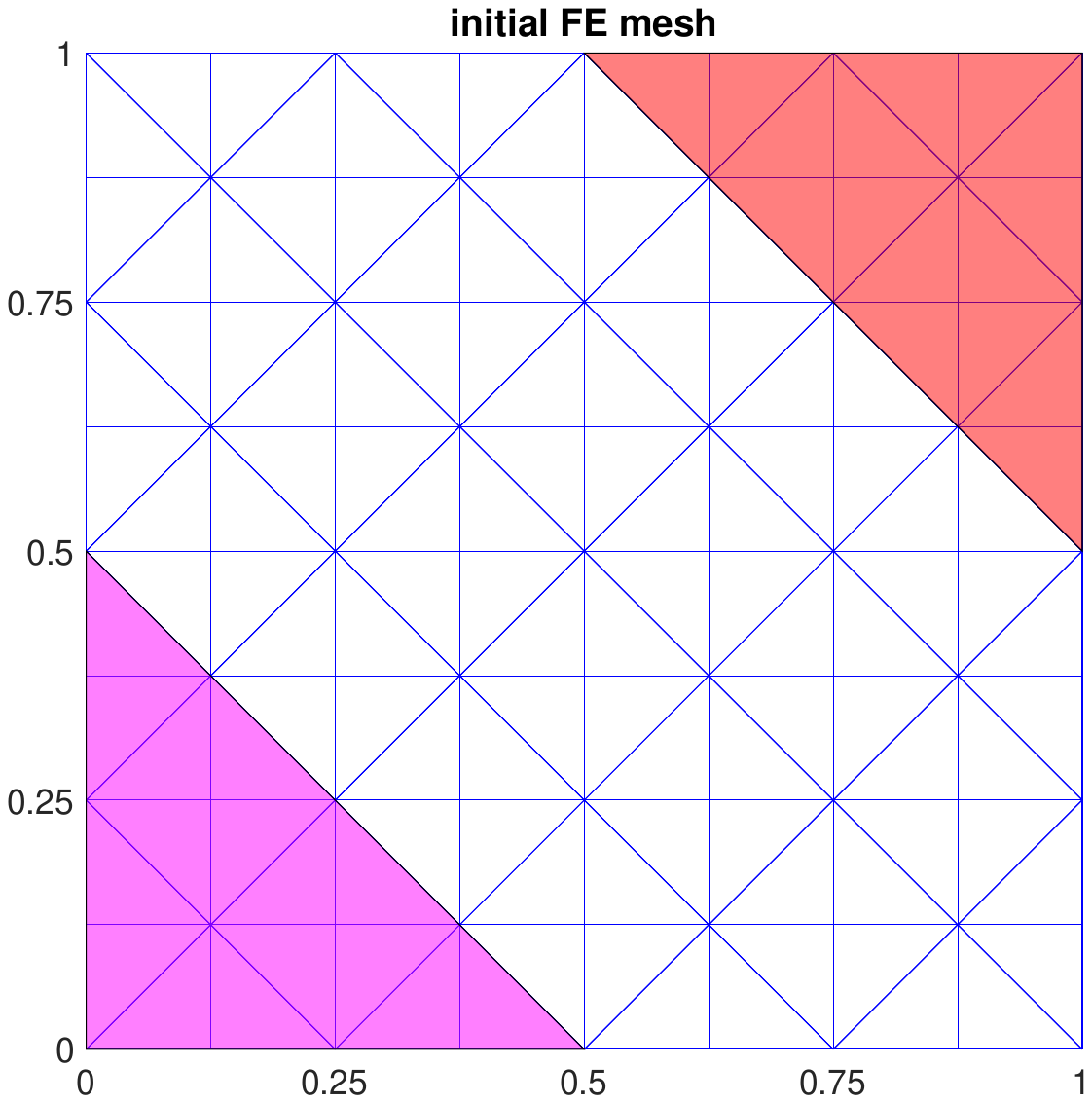}
\end{subfigure}
\begin{subfigure}{.24\textwidth}
    \centering
    \caption*{Setup 4}
    \includegraphics[width = .95\linewidth,  trim=0.5cm 5.5cm 1cm 5.6cm,clip]{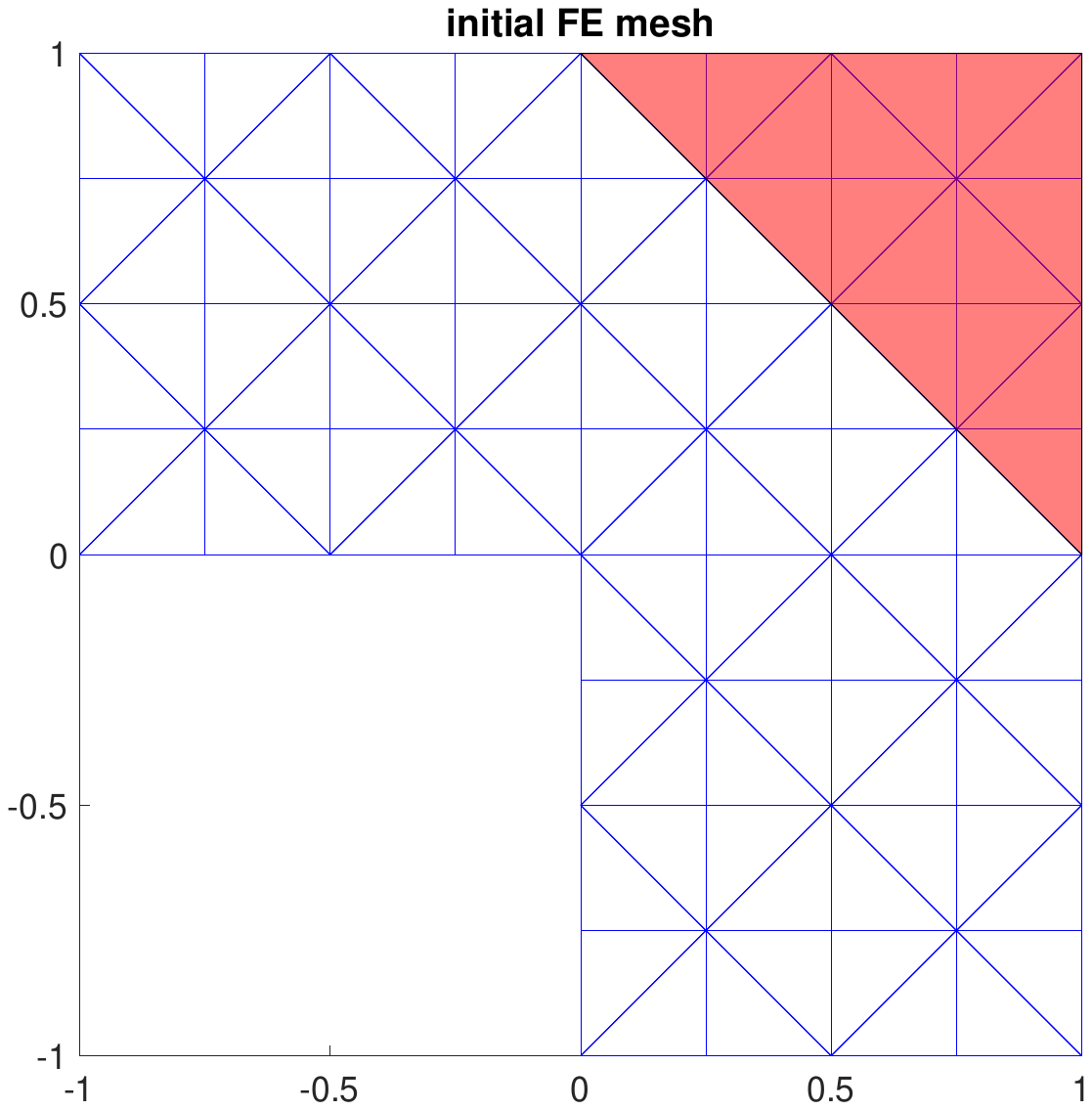}
\end{subfigure}
\caption{The supports of spatial features and initial meshes for all four setups.} 
\label{fig:meshes:init}
\end{figure}

\smallskip

\emph{Setup 2: estimation of pointwise values.}
Let $D_\delta = (-1, 1)^2 \setminus \overline{T}_\delta$, where $T_\delta$ is the triangle with vertices $(0, 0)$, $(-1, \delta)$, and $(-1, -\delta)$.
In this test case, we consider the spatial slit domain $D = (-1, 1)^2 \setminus ([-1, 0] \times \{0\})$.
While the slit domain is not Lipschitz, it is well known that an elliptic problem on this domain
can be seen as the limit of the problems posed on Lipschitz domains $D_\delta$ as $\delta \to 0$.
Therefore, for this setup, we will perform computations on the domain $D = D_\delta$ with $\delta = 0.005$.
Aiming to approximate the expectation of a pointwise value of the solution $u$ to problem~\eqref{eq:pde:strong},
we consider the following linear goal functional:
\begin{equation*}
    \Q(u) = \int_{\G}\int_D q(x) u(x,\y) \dx \, \dpi(\y),
\end{equation*}
where the weight function $q \in L^{\infty}(D)$ is a mollifier centered at the point $x_0 = (0.4, -0.5) \in D$ with radius $r = 0.15$
(we refer to equation~(58) in~\cite{bprr18+} for the specific expression for the function $q(x)$).
The support of $q(x)$ and the initial mesh on $D$ are depicted in Figure~\ref{fig:meshes:init}.
In this setup, we consider
\revision{the constant right-hand side function $f(x) \equiv 1$ and
a nonaffine parametric diffusion coefficient 
$a(x,\y) = \exp(h(x,\y))$, where $h(x,\y)$ is given by~\eqref{kll} with $M=4$.}

\smallskip

\emph{Setup 3: second moment of a linear goal functional.}
In this test case, we set the spatial domain to be the unit square~$D = (0, 1)^2$, and
we consider again the affine-parametric diffusion coefficient \revision{with 4 random parameters},
i.e., $a(x,\y) = h(x, \y)$  \revision{and $M=4$}.
Furthermore, we follow~\cite{ms2009, bprr18+} and choose the forcing term in~\eqref{eq:pde:strong}
so that the functional on the right-hand side of~\eqref{eq:weak:primal} is given by
\begin{align} \label{eq:MS:example:RHS}
    \F(v) = 
    - \int_\G \int_{D} (\chi_{T_f}, 0) \cdot \nabla v(x,\y) \dx\, \dpi(\y)
    =
    -\int_\G \int_{T_f} \frac{\partial v}{\partial x_1} (x,\y) \dx\, \dpi(\y)
    \ \, \forall\, v \in \V;
\end{align}
here $\chi_{T_f} : D \to \{0, 1\}$ is the characteristic function of the 
triangle $T_f \subset D$ with vertices $(0, 0)$, $(0.5, 0)$, and $(0, 0.5)$. 
We consider a nonlinear goal functional given by the (scaled) second moment of a linear quantity of interest.
Specifically, we define
\begin{equation} \label{eq:MS:example:goal}
    \Q(u) = 100 \int_{\G} \left( \int_D q(x) u(x, \y) \, \dx\right)^2 \dpi(\y), 
\end{equation}
where $q(x) = \chi_{T_q}/|T_q|$ and
$\chi_{T_q}:D \to \{0, 1\}$ is the characteristic function of the triangle $T_q$ with vertices $(0.5, 1)$, $(1, 0.5)$, and $(1, 1)$.
The motivation behind the representation for $\F(v)$ in~\eqref{eq:MS:example:RHS} 
is to introduce non-geometric singularities 
into the solution to the primal problem,
so that the primal and the dual solutions
exhibit singular behavior in two different regions of the computational domain.
Figure~\ref{fig:meshes:init} shows the initial mesh on $D$ as well as the triangular subdomains $T_f$ (magenta) and $T_q$ (red).

In order to calculate $\Q(\scsol)$, we proceed as follows:
\begin{align*}
    \frac{\Q(\scsol)}{100} & = \int_\G \bigg(\int_D q(x) \scsol(x, \y) \, \dx\bigg)^2 \,\dpi(\y) 
    \\[4pt]
    & = \int_\G \bigg(\int_D q(x) \sum_{\ybar \in \Colpts_\bullet} \udotyh(x) \LagrBasis{\bullet \ybar}{}(\y) \, \dx \bigg) \,
          \bigg(\int_D q(x) \sum_{\ybar' \in \Colpts_\bullet} \udotyhh(x) \LagrBasis{\bullet \ybar'}{}(\y) \, \dx \bigg) \, \dpi(\y) 
    \\[4pt]
    & = \sum_{\ybar,  \ybar' \in \Colpts_\bullet} \int_D q(x) \udotyh(x) \, \dx  \int_D q(x) \udotyhh(x) \, \dx \int_{\G} \LagrBasis{\bullet \ybar}{}(\y) \LagrBasis{\bullet \ybar'}{}(\y) \, \dpi(\y).
\end{align*}
The integrals over the parameter domain $\G$ are calculated as explained in~\S\ref{sec:approx:goal}.
The spatial integrals are evaluated in a standard way by representing the samplewise finite element approximations as, e.g.,
$\udotyh(x) = \sum_{\xi \in \NN_\bullet} c_{\bullet \ybar, \xi} \varphi_{\bullet,\xi}(x)$ and calculating
the $q(x)$-weighted spatial integrals of the finite element basis functions $\varphi_{\bullet,\xi}(x)$.

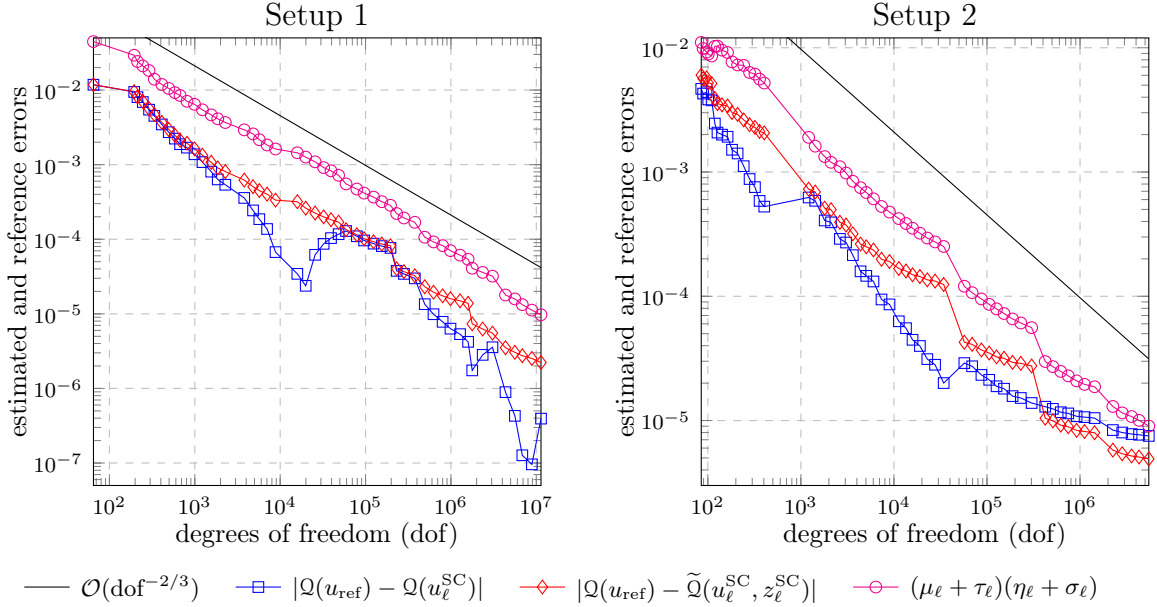
\begin{figure}[!thp]
\begin{tikzpicture}
\pgfplotstableread{affine_square_subdomain_lin_v1.dat}{\one}
\begin{loglogaxis}
[
width = 7.5cm, height = 7.5cm,						
title = Setup 1,
title style={font=\fontsize{12pt}{14pt}\selectfont},
xlabel={degrees of freedom (dof)}, 					
xlabel style={font=\fontsize{10pt}{12pt}\selectfont},		
ylabel={estimated and reference errors},				
ylabel style={font=\fontsize{10pt}{12pt}\selectfont}, 		
ymajorgrids=true, xmajorgrids=true, grid style=dashed,	
xmin = 65, xmax = 11356902,						
ymin = 5*10^(-8),	 ymax = 0.05,							
legend columns=4,
legend entries={
{$\mathcal{O}(\text{dof}^{-2/3})\quad$},
{$|\Q(u_{\text{ref}}) - \Q(u_{\ell}^{\rm SC})|\quad$},
{$|\Q(u_{\text{ref}}) - \Qbar(u_{\ell}^{\rm SC}, z_{\ell}^{\rm SC})|\quad$},
{$(\mu_{\ell} + \tau_{\ell})(\eta_{\ell} + \sigma_{\ell})\quad$}
},
legend to name=commonlegend1,
legend style={legend cell align=left, row sep=5pt, column sep=2pt, fill=none, draw=none, font={\fontsize{9pt}{12pt}\selectfont}},
]
\addplot[black]		table[x=dofs, y=straight_line]{\one};
\addplot[blue,mark=square,mark size=2pt]		table[x=dofs, y=error_unc]{\one};
\addplot[red,mark=diamond,mark size=2.6pt]		table[x=dofs, y=error_corr]{\one};
\addplot[magenta,mark=o,mark size=2.2pt]		table[x=dofs, y=error_d]{\one};

\end{loglogaxis}
\end{tikzpicture}
\quad
\begin{tikzpicture}
\pgfplotstableread{exp_mollifier_lin_v1.dat}{\two}
\begin{loglogaxis}
[
width = 7.5cm, height = 7.5cm,						
title = Setup 2,
title style={font=\fontsize{12pt}{14pt}\selectfont},
xlabel={degrees of freedom (dof)}, 					
xlabel style={font=\fontsize{10pt}{12pt}\selectfont},		
ylabel={estimated and reference errors},				
ylabel style={font=\fontsize{10pt}{12pt}\selectfont}, 		
ymajorgrids=true, xmajorgrids=true, grid style=dashed,	
xmin = 85, xmax = 5483302,						
ymin = 3*10^(-6),	 ymax = 0.012,							
]
\addplot[black]		table[x=dofs, y=straight_line]{\two};
\addplot[blue,mark=square,mark size=2pt]		table[x=dofs, y=error_unc]{\two};
\addplot[red,mark=diamond,mark size=2.6pt]		table[x=dofs, y=error_corr]{\two};
\addplot[magenta,mark=o,mark size=2.2pt]		table[x=dofs, y=error_d]{\two};

\end{loglogaxis}
\end{tikzpicture}
\ref{commonlegend1}
\caption{Evolution of the goal-oriented error estimates (magenta lines) as well as the reference errors for uncorrected (blue lines)
and corrected (red lines) approximations of $\Q(u)$ for setups~1 and~2.}
\label{fig:results:lin}
\end{figure}

\begin{figure}[!thp]

\begin{tikzpicture}
\pgfplotstableread{affine_second_moment_nl_v1.dat}{\one}
\begin{loglogaxis}
[
width = 7.5cm, height = 7.5cm,						
title = Setup 3,
title style={font=\fontsize{12pt}{14pt}\selectfont},
xlabel={degrees of freedom (dof)}, 					
xlabel style={font=\fontsize{10pt}{12pt}\selectfont},		
ylabel={estimated and reference errors},				
ylabel style={font=\fontsize{10pt}{12pt}\selectfont}, 		
ymajorgrids=true, xmajorgrids=true, grid style=dashed,	
xmin = 81, xmax = 1997406,						
ymin = 3*10^(-9),	 ymax = 0.011,							
legend columns=4,
legend entries={
{$\mathcal{O}(\text{dof}^{-2/3})\quad$},
{$|\Q(u_{\text{ref}}) \,{-}\, \Q(u_{\ell}^{\rm SC})|\quad$},
{$|\Q(u_{\text{ref}}) \,{-}\, \Qbar(u_{\ell}^{\rm SC}, z_{\ell}^{\rm SC})|\quad$},
{$(\mu_{\ell} \,{+}\, \tau_{\ell})(\mu_{\ell} \,{+}\, \tau_{\ell} \,{+}\, \eta_{\ell} \,{+}\, \sigma_{\ell})\quad$},
{$|\Q(u_{\text{ref}}) \,{-}\, \Q(u_{\ell}^{\rm SC})|,~\cite{bsx22}$}
},
legend to name=commonlegend2,
legend style={legend cell align=left, row sep=5pt, column sep=2pt, fill=none, draw=none, font={\fontsize{9pt}{12pt}\selectfont}},
]
\addplot[black]		table[x=dofs, y=straight_line]{\one};
\addplot[blue,mark=square,mark size=2pt]		table[x=dofs, y=error_unc]{\one};
\addplot[red,mark=diamond,mark size=2.6pt]		table[x=dofs, y=error_corr]{\one};
\addplot[magenta,mark=o,mark size=2.2pt]		table[x=dofs, y=error_d]{\one};

\end{loglogaxis}
\end{tikzpicture}
\quad
\begin{tikzpicture}
\pgfplotstableread{exp_convection_nl_M_10.dat}{\two}
\begin{loglogaxis}
[
width = 7.5cm, height = 7.5cm,						
title = Setup 4,
title style={font=\fontsize{12pt}{14pt}\selectfont},
xlabel={degrees of freedom (dof)}, 					
xlabel style={font=\fontsize{10pt}{12pt}\selectfont},		
ylabel={estimated and reference errors},				
ylabel style={font=\fontsize{10pt}{12pt}\selectfont}, 		
ymajorgrids=true, xmajorgrids=true, grid style=dashed,	
xmin = 65, xmax = 1871987,						
ymin = 3*10^(-7),	 ymax = 0.01,							
]
\addplot[black]		table[x=dofs, y=straight_line]{\two};
\addplot[blue,mark=square,mark size=2pt]		table[x=dofs, y=error_unc]{\two};
\addplot[red,mark=diamond,mark size=2.6pt]		table[x=dofs, y=error_corr]{\two};
\addplot[magenta,mark=o,mark size=2.2pt]		table[x=dofs, y=error_d]{\two};

\end{loglogaxis}
\end{tikzpicture}
\ref{commonlegend2}
\caption{Evolution of the goal-oriented error estimates (magenta lines) as well as the reference errors for uncorrected (blue lines)
and corrected (red lines) approximations of $\Q(u)$ for setups~3 and~4.}
\label{fig:results:nl}
\end{figure}
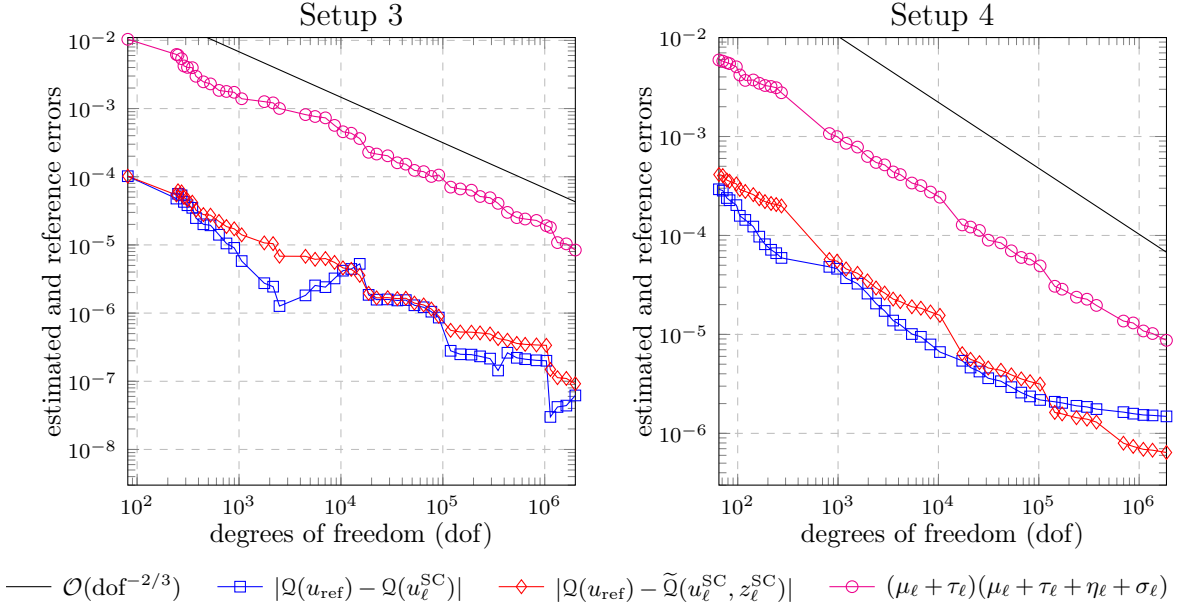

\emph{Setup 4:  expectation of a nonlinear convection term}.
Here, as in the first test case, we consider the model problem on the L-shaped domain $D = (-1, 1)^2 \setminus (-1, 0]^2$
and set $f(x) \equiv 1$.
We consider a nonaffine parametric coefficient $a(x, \y) = \exp(h(x,y))$ and define the goal functional as
the expectation of a nonlinear convection term, i.e.,
\begin{equation*}
    \Q(u) = \int_{\G} \int_D q(x) u(x, \y) \left(\frac{\partial u}{\partial x_1}(x, \y) +  \frac{\partial u}{\partial x_2}(x, \y)\right) \; \dx \;\dpi(\y),
\end{equation*}
where $q(x) = \chi_T/|T|$ and $\chi_T: D \to \{0, 1\}$ is the characteristic
function of the triangle $T \,{\subset}\, D$ with vertices  $(0,1)$, $(1, 0)$, $(1, 1)$;
see Figure~\ref{fig:meshes:init} that also shows the initial mesh on~$D$.
\revision{This time, to demonstrate the robustness of the algorithm with respect to the number of random parameters, we increase the number of Fourier modes in~\eqref{kll} to $M = 10$.}

In order to evaluate $\Q(\scsol)$, we use the following formula:
\begin{align*}
    \Q(\scsol) & =
    \int_\G \int_D q(x) \scsol(x, \y) \bigg(\frac{\partial \scsol}{\partial x_1}(x, \y) +  \frac{\partial \scsol}{\partial x_2}(x, \y)\bigg) \, \dx \, \dpi(\y) 
    \\[4pt]
    & = \int_{\G} \int_D q(x) \bigg[\sum_{\ybar \in \Colpts_\bullet} \udotyh(x) \LagrBasis{\bullet \ybar}{}(\y) \bigg]
           \sum_{\ybar' \in \Colpts_\bullet}\bigg(\frac{\partial \udotyhh}{\partial x_1}(x) +  \frac{\partial \udotyhh}{\partial x_2}(x)\bigg)
           \LagrBasis{\bullet \ybar'}{}(\y) \, \dx \, \dpi(\y)
    \\[4pt]
    & = \sum_{\ybar,  \ybar' \in \Colpts_\bullet}
          \int_D q(x)\udotyh(x) \bigg(\frac{\partial \udotyhh}{\partial x_1}(x) +  \frac{\partial \udotyhh}{\partial x_2}(x)\bigg) \, \dx
          \int_{\G} \LagrBasis{\bullet \ybar}{}(\y) \LagrBasis{\bullet \ybar'}{}(\y) \, \dpi(\y).
\end{align*}
Here, the computation of integrals over the parameter domain $\G$ follows again the procedure described in~\S\ref{sec:approx:goal},
whereas the spatial integrals can be seen as bilinear forms.
Thus, we evaluate these spatial integrals as matrix-vector products
$\mathbf{c}_{\bullet\ybar}^{\sf T} {\mathbf{B}}_\bullet \mathbf{c}_{\bullet\ybar'}$,
where $\mathbf{c}_{\bullet\ybar},\,\mathbf{c}_{\bullet\ybar'}$ are the coefficient vectors for Galerkin approximations
$\udotyh,\, \udotyhh$ and
$[ {\mathbf{B}}_\bullet ]_{\xi, \eta \in \NN_\bullet} = 
  \int_D q(x) \varphi_{\bullet,\xi}(x)
                    \big(\frac{\partial\varphi_{\bullet,\eta}}{\partial x_1} (x) + \frac{\partial \varphi_{\bullet,\eta} }{\partial x_2 } (x) \big) \, \dx
$
(note that, at a given iteration of the algorithm, the matrix ${\mathbf{B}}_\bullet$ is the same for all collocation points).

\begin{figure}[!thp]
\begin{subfigure}{.48\textwidth}
    \centering
    \caption*{Setup 1 ($\TT_{32}$)}
    \includegraphics[width = .95\linewidth,  trim=1.6cm 6.3cm 2cm 6.4cm,clip]{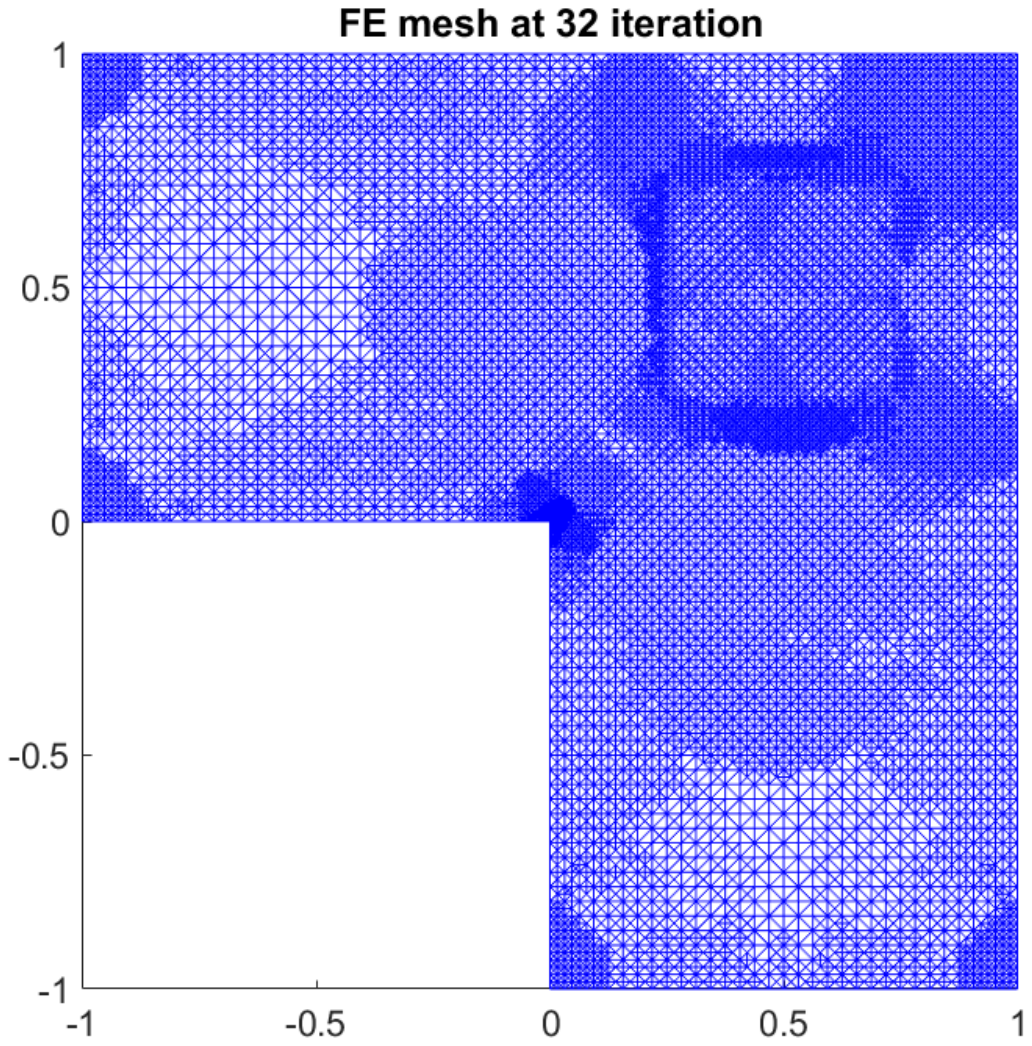}
\end{subfigure}
\begin{subfigure}{.48\textwidth}
    \centering
    \caption*{Setup 2 ($\TT_{30}$)}
    \includegraphics[width = .95\linewidth,  trim=1.6cm 6.3cm 2cm 6.4cm,clip]{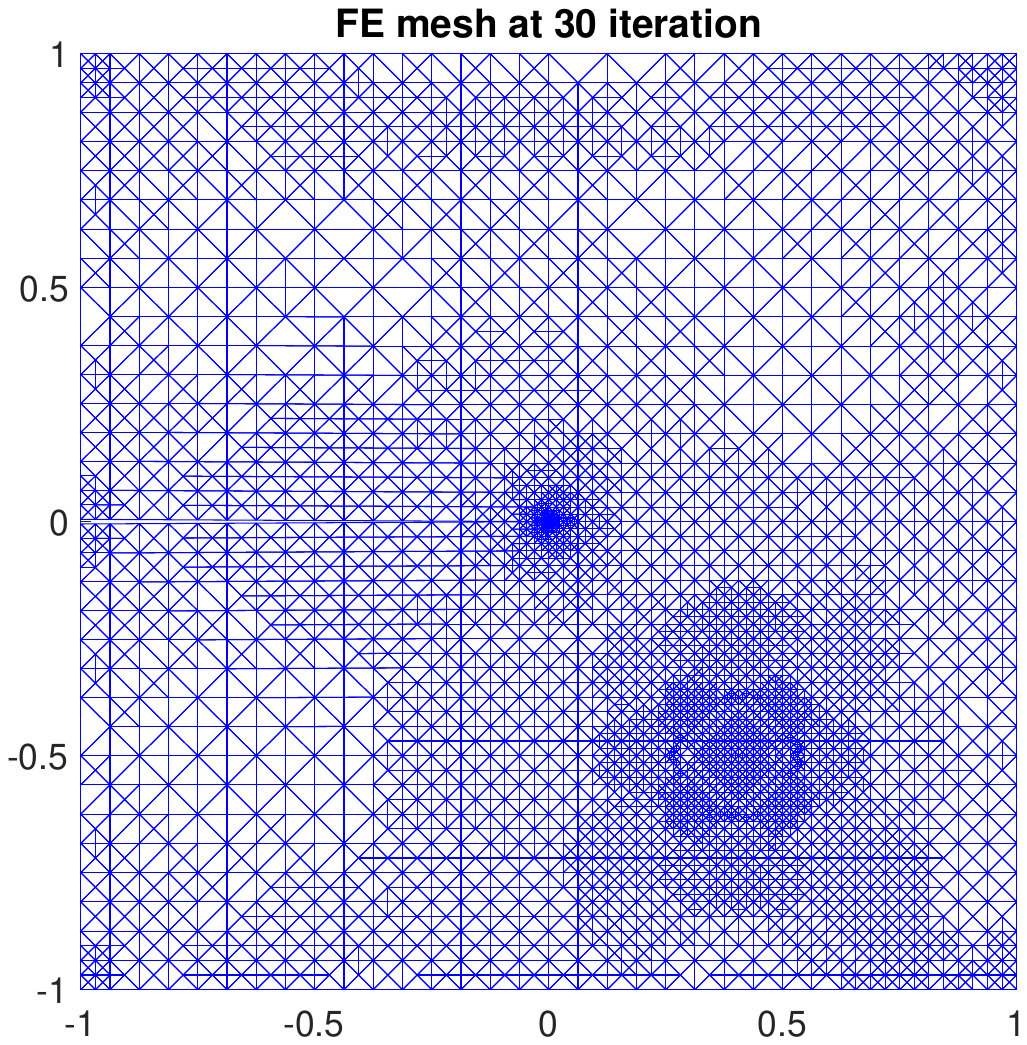}
\end{subfigure}
\begin{subfigure}{.48\textwidth}
    \centering
    \caption*{Setup 3 ($\TT_{30}$)}
    \includegraphics[width = .95\linewidth, trim=1.6cm 6.3cm 2cm 6.4cm,clip]{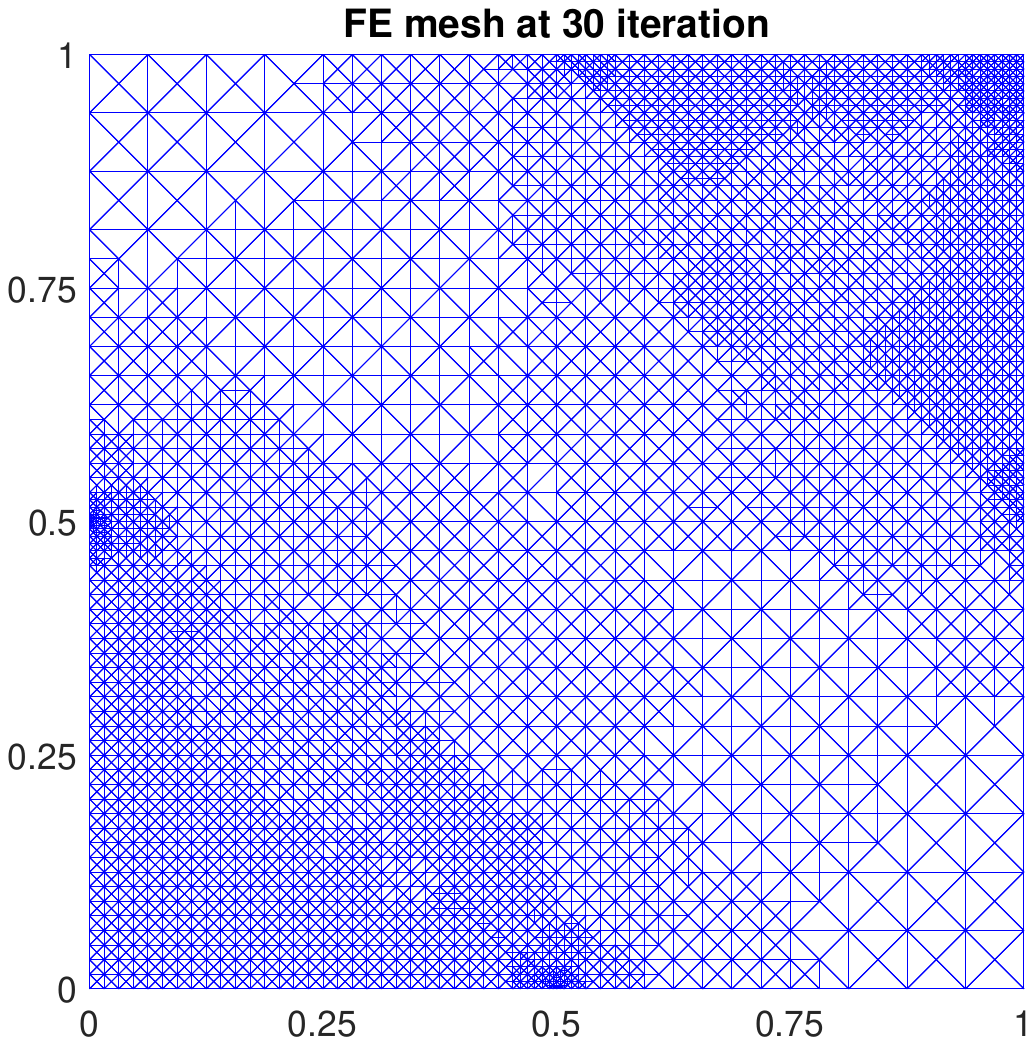}
\end{subfigure}
\begin{subfigure}{.48\textwidth}
    \centering
    \caption*{Setup 4 ($\TT_{27}$)}
    \includegraphics[width = .95\linewidth, trim=1.6cm 6.3cm 2cm 6.4cm,clip]{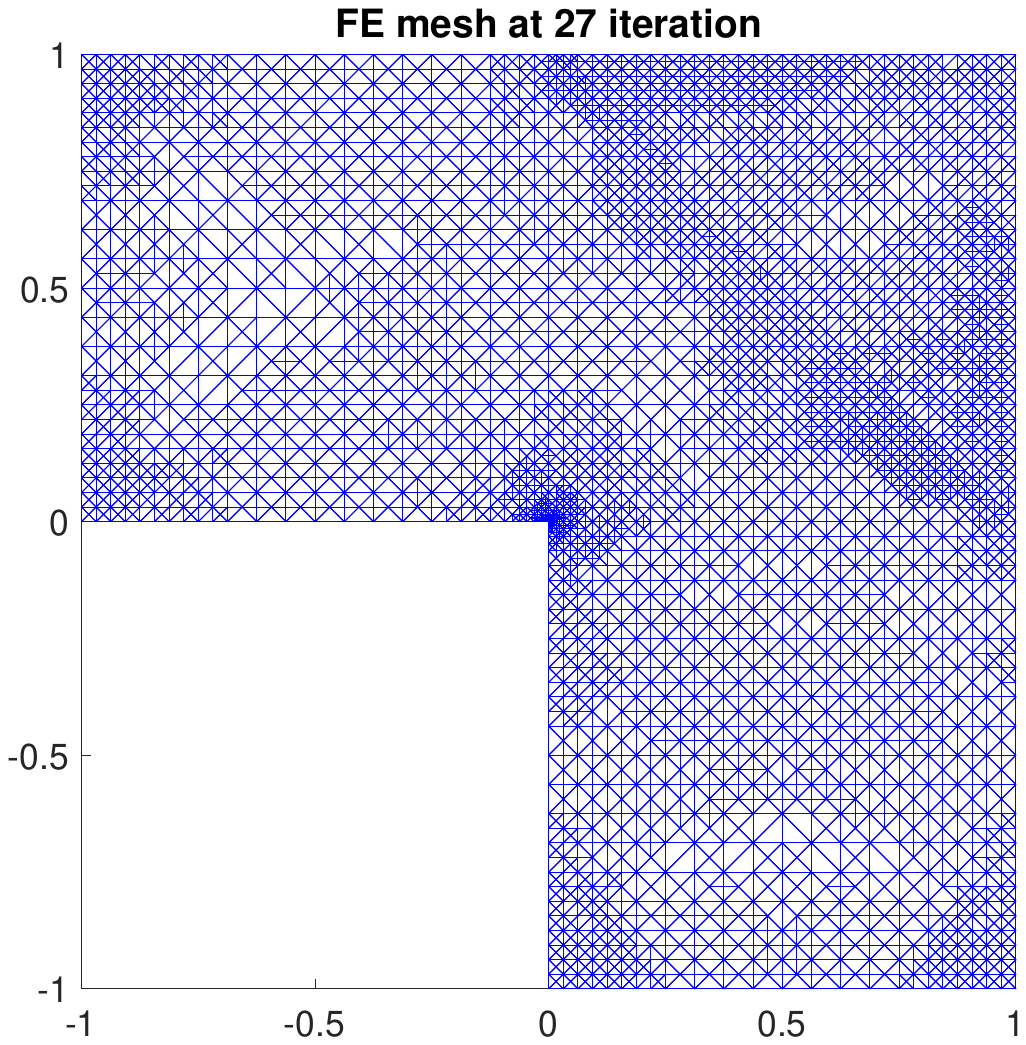}
\end{subfigure}
\caption{Adaptively refined meshes for all four setups.}
\label{fig:meshes}
\end{figure}

In Figures~\ref{fig:results:lin} and~\ref{fig:results:nl}, for all setups, we plot the goal-oriented error estimates
given by the right-hand sides of~\eqref{eq:goal:error:estimate} and~\eqref{eq:goal:error:estimate:nl}, respectively,
as well as the reference errors for \emph{uncorrected}
and \emph{corrected} approximations of $\Q(u)$
versus the number of degrees of freedom~(dof) at each iteration of Algorithm~\ref{algorithm}. 
The reference value $\Q(u_{{\rm ref}})$ is calculated using the reference solution $u_{\rm ref}$ to the corresponding  test problem.
The reference solution for each setup is computed as an SC-FEM approximation using a large set of collocation points
(generated by the union of the final index set $\indset_{\ell_*}$ and its full margin $\marg(\indset_{\ell_*})$) and
the second-order~(P2) FEM on the uniform refinement of the final mesh $\TT_{\ell_*}$ generated by the adaptive algorithm
for the corresponding test problem.
\revision{The average computational time for running Algorithm~\ref{algorithm} in our experiments is about 20\% of the time required to obtain $u_{{\rm ref}}$. Furthermore, the runtime of Algorithm~\ref{algorithm} can be reduced significantly if the calculation of the error estimates~\eqref{eq:spatial:estimate},~\eqref{eq:param:estimate}, and~\eqref{eq:dual:estimates} is performed periodically.}
The plots show that in each setup the adaptive algorithm drives the goal-oriented error estimates (magenta lines) \revision{\emph{towards}} zero.
Moreover, these error estimates and the reference errors in the corrected approximation of $\Q(u)$
decay with the rate $O(\text{dof}^{-2/3})$. 
As expected, this is twice the rate of decay of the error in adaptive SC-FEM approximations of the primal solution,
cf.~\cite[section~5]{bsx22}.
In addition, the introduction of the correction term in~\eqref{eq:goal:approx} has a stabilizing effect on the decay of the corresponding reference error.
Indeed, for each setup, the reference errors for corrected approximations of the goal functional
(red lines in Figures~\ref{fig:results:lin},~\ref{fig:results:nl})
decay monotonically (with exception of maybe a few initial iterations) and do so with the same rate as the goal-oriented error estimates.
In contrast, the reference errors for uncorrected approximations (blue lines in Figures~\ref{fig:results:lin},~\ref{fig:results:nl})
either exhibit jumps (in setups~1 and~3) or deviate from the decay rate of the goal-oriented error estimates (in setups~2 and~4).

We note that the gap between the red and magenta lines (that reflects the size of effectivity indices for the goal-oriented error estimation)
for nonlinear goal functionals in Figure~\ref{fig:results:nl} is much wider than that for linear goal functionals in Figure~\ref{fig:results:lin}.
Specifically, the average effectivity index (i.e., the ratio of the error estimate to the reference error for the corrected approximation of the goal functional)
is 7.1 in setup~1,
2.4 in setup~2,
100.2 in setup~3, and
16.2 in setup~4.
Larger effectivity indices in the case of nonlinear goal functionals are likely due to the size of the constant $C_{\rm goal}$
in assumption~\eqref{ineq:gateaux} that affects the (hidden) constant in the goal-oriented error estimate~\eqref{eq:goal:error:estimate:nl}. 
\revision{We have observed that the average effectivity index remains the same with the increase in $M$:
we have carried out the experiment described in Setup~4 for both $M=4$ and $M=10$ and have not seen
noticeable changes in effectivity indices.
This is consistent with the results in the non-goal-oriented setting
as reported in~\cite[section~5]{bsx22}.
Overall, our adaptive algorithm performs well for problems with a moderate number
of random parameters, i.e., for~$M \le 50$.}

\revision{The use of a quadrature for the evaluation of the bilinear form $\B(\scsol,\scsoldual)$ typically admits its own source of error.
The experiments performed in~\cite[Chapter~4]{Round_PhD_Thesis} have shown
that the error associated with the quadrature approximation is dominated by the error in the QoI.
However, it appears that our approximations of the QoI converge more quickly than the quadrature approximations, meaning that when running the adaptive algorithm for smaller tolerances
(that correspond to $10^7$ or more degrees of freedom in the SC-FEM approximation),
a more accurate quadrature rule should be preferred.
}

Figure~\ref{fig:meshes} depicts locally refined finite element meshes generated at intermediate iterations of the adaptive algorithm.
This figure shows that for all setups 
the algorithm refines the finite element mesh according to spatial features in both the primal and the dual solution.
These include corner singularities, local features due to a spatial subdomain in the definition of the goal functional (cf. Figure~\ref{fig:meshes:init})
and, in setup~3, non-geometric singularities induced by the choice of
the forcing term~$f$ in the primal problem.

\section{Concluding remarks} \label{sec:conclusions}

Surrogate approximations of solutions to PDEs with uncertain inputs play an important role in reliable uncertainty quantification.
They can be cheaply evaluated at any point in the parameter domain and are used to estimate a wide range of QoIs,
thus, providing an effective alternative to Monte Carlo simulations.
While being effective in approximating the input-output map and QoIs, surrogate approximations introduce numerical errors that
need to be estimated and controlled in the course of simulations.
In this work, we have proved reliable a posteriori estimates for the errors associated with computing QoIs in the SC-FEM context.
Furthermore, we have designed a goal-oriented adaptive algorithm for reducing these errors.
\revision{The key features of our algorithm are its foundations in rigorous a posteriori analysis and its use of
\emph{adaptively refined} (in contrast to \emph{a priori chosen}) approximation components.}
\revision{The error analysis and the proposed adaptive algorithm can easily be extended to more general elliptic PDEs with random inputs
represented using a finite number of bounded random parameters;
all that is required is the well-posedness of the corresponding
samplewise weak formulation and the norm equivalence (cf.~\eqref{eq:norm:equiv})
in appropriate function spaces.}
The numerical experiments presented in this paper and more extensive experiments included in~\cite{Round_PhD_Thesis}
demonstrate the robustness of our error estimation strategy and the effectiveness of the adaptive algorithm.
In particular, our conclusions hold for PDE problems with affine or nonaffine parametrizations of uncertainty in PDE inputs
and for a range of QoIs represented by linear or nonlinear goal functionals.
%

Adaptive algorithms for the sparse grid-based SC-FEM are particularly effective when combined with 
dimensionality reduction techniques (see, e.g.,~\cite{FarcasSBNU_18_MAS, KonradFPDiSJNB_22_DDL})
and/or dimension adaptivity
\revision{(see, e.g.,~\cite{SchillingsSchwab2013, ChkifaCohenSchwab2014, ErnstST_18_CSC, FarcasGBJN_20_SDA, GriebelS_24_DAC})}.
\revision{These techniques become critical when tackling PDE problems
with a large or countably infinite number of unbounded random parameters
(e.g., PDEs with inputs represented by lognormal random fields) that are not covered in this work.}
The application of these techniques \revision{in such a setting and} within the goal-oriented adaptive framework proposed in our work is
currently being investigated and will be the subject of a future publication.

\bibliographystyle{siam}
\bibliography{references}

\end{document}